\documentclass[a4paper, 11pt,reqno]{amsart}

\usepackage[utf8]{inputenc}
\usepackage[T1]{fontenc}
\usepackage{mathtools, amsthm, amsfonts, amssymb, amsmath}
\usepackage{microtype}

\usepackage{latexsym}
\usepackage{braket}
\usepackage{thmtools}
\usepackage{mathrsfs}
\usepackage{textcomp}
\usepackage{graphicx}
\usepackage{setspace}
\usepackage{nicefrac}
\usepackage{enumerate}
\usepackage{wasysym}
\usepackage[normalem]{ulem}
\usepackage{upgreek}

\usepackage{paralist}
\usepackage{xcolor}
\usepackage{etoolbox}
\usepackage{mathdots}
\usepackage{wrapfig}
\usepackage{floatflt}
\usepackage{tensor} 
\usepackage{lscape}
\usepackage{stmaryrd}
\usepackage[enableskew]{youngtab}
\usepackage{ytableau}
\usepackage{pifont}

\usepackage{nicematrix}

\usepackage{tikz}
\usepackage{tikz-cd}
\usepackage{tikz-3dplot}
\usetikzlibrary{arrows}
\usetikzlibrary{matrix}
\usetikzlibrary{positioning}
\usetikzlibrary{calc}
\usetikzlibrary{cd}

\usepackage{hyperref}
\hypersetup{
    colorlinks,
    linkcolor={red!50!black},
    citecolor={blue!50!black},
    urlcolor={blue!80!black}
}

\usepackage{enumitem}
\usepackage{nicematrix} 

\usepackage{physics} 

\usepackage{listings}
\usepackage{comment}

\usepackage[all,cmtip]{xy}
\usepackage[labelfont=bf, font={small}]{caption}[2005/07/16]

\usepackage[nameinlink]{cleveref}
\usepackage{xcolor}
\definecolor{red}{rgb}{1,0,0}

\newcommand{\vvirg}{ , \dots , }
\newcommand{\ootimes}{ \otimes \cdots \otimes }

\newcommand{\contract}{\rotatebox[origin=c]{180}{ \reflectbox{$\neg$} }}

\newcommand{\bfJ}{\mathbf{J}}

\newcommand{\calH}{\mathcal{H}}

\newcommand{\calO}{\mathcal{O}}

\newcommand{\calR}{\mathcal{R}}
\newcommand{\calS}{\mathcal{S}}

\newcommand{\bbC}{\mathbb{C}}

\newcommand{\bbN}{\mathbb{N}}

\newcommand{\bbP}{\mathbb{P}}

\newcommand{\bbS}{\mathbb{S}}

\newcommand{\frakS}{\mathfrak{S}}

\newcommand{\frakg}{\mathfrak{g}}

\renewcommand{\phi}{\varphi}
\newcommand{\eps}{\varepsilon}

\newcommand{\dashto}{\dashrightarrow}

\renewcommand{\bar}[1]{\overline{#1}}

\newcommand{\id}{\mathrm{id}}

\DeclareMathOperator{\codim}{codim}

\DeclareMathOperator{\End}{End}

\DeclareMathOperator{\Sym}{Sym}

\newcommand{\SO}{\mathrm{SO}}

\newcommand{\bfVec}{\mathbf{Vec}}
\newcommand{\bfVar}{\mathbf{Var}}

\newcommand{\Sing}{\mathrm{Sing}}

\newcommand{\slrk}{\mathrm{slrk}}
\newcommand{\str}{\mathrm{str}}
\newcommand{\ustr}{\underline{\mathrm{str}}}

\newcommand{\Gr}{\mathrm{Gr}}

\DeclareMathAccent{\wtilde}{\mathord}{largesymbols}{"65}

\newcommand{\GL}{\mathrm{GL}}

\newcommand{\SL}{\mathrm{SL}}
\newcommand{\fraksl}{\mathfrak{sl}}
\newcommand{\frakgl}{\mathfrak{gl}}
\newcommand{\frakso}{\mathfrak{so}}

\newcommand{\diff}{\mathrm{d}}

\DeclareMathOperator{\Stab}{Stab}

\usepackage[top=3cm, bottom=3cm, left=3.3cm, right =3.3cm]{geometry}

\numberwithin{equation}{section}

\newtheorem{theorem}[equation]{Theorem}
\newtheorem{lemma}[theorem]{Lemma}
\newtheorem{proposition}[theorem]{Proposition}
\newtheorem{corollary}[theorem]{Corollary}
\newtheorem{conjecture}[theorem]{Conjecture}

\theoremstyle{definition}
\newtheorem{definition}[theorem]{Definition}

\newtheorem*{question*}{Question}

\newtheorem{remarkx}[theorem]{Remark}
\newenvironment{remark}
  {\pushQED{\qed}\remarkx}
  {\popQED\endremarkx}

\crefname{theorem}{theorem}{theorems}
\Crefname{theorem}{Theorem}{Theorems}
\crefname{lemma}{lemma}{lemmas}
\Crefname{lemma}{Lemma}{Lemmas}
\crefname{proposition}{proposition}{propositions}
\Crefname{proposition}{Proposition}{Propositions}
\crefname{corollary}{corollary}{corollaries}
\Crefname{corollary}{Corollary}{Corollaries}
\crefname{definition}{definition}{definitions}
\Crefname{definition}{Definition}{Definitions}
\crefname{remark}{remark}{remarks}
\Crefname{remark}{Remark}{Remarks}
\crefname{remarkx}{remark}{remarks}
\Crefname{remarkx}{Remark}{Remarks}
\crefname{example}{example}{examples}
\Crefname{example}{Example}{Examples}
\crefname{examplex}{example}{examples}
\Crefname{examplex}{Example}{Examples}

\title[Small slice rank and strength]{Polynomials of small slice rank and strength}

\author[C. Flavi]{Cosimo Flavi}
\author[F. Gesmundo]{Fulvio Gesmundo}
\author[A. Oneto]{Alessandro Oneto}
\author[E. Ventura]{Emanuele Ventura}

\address[C. Flavi]{Wydział Matematyki, Informatyki i Mechaniki, Uniwersytet Warszawski, ul.~Stefana Banacha 2, 02-097 Warsaw, Poland.}
\email{c.flavi@uw.edu.pl}

\address[F. Gesmundo]{Institut de Mathématiques de Toulouse; UMR5219 -- Université de Toulouse; CNRS -- UPS, F-31062 Toulouse Cedex 9, France}
\email{fgesmund@math.univ-toulouse.fr}

\address[A. Oneto]{Università degli Studi di Trento, Dipartimento di Matematica, Via Sommarive 14, 38123 – Povo (Trento), Italy}
\email{alessandro.oneto@unitn.it}

\address[E. Ventura]{Politecnico di Torino, Dipartimento di Scienze Matematiche ``G. L. Lagrange'', Corso Duca degli Abruzzi 24, 10129 Torino, Italy}
\email{emanuele.ventura@polito.it}

\begin{document}

\begin{abstract}
This paper investigates defining equations for secant varieties of the variety of reducible polynomials, which geometrically encode the notions of strength and slice rank of homogeneous polynomials. We present three main results. First, we reinterpret Ruppert's classical equations for reducible ternary forms in the language of representation theory and we extend them to an arbitrary number of variables. Second, we construct new determinantal equations for polynomials of small strength based on syzygies of their partial derivatives. Finally, we establish a reduction theorem for cubic forms, proving that slice rank two is determined by generic linear sections in $14$ variables; this gives one of the few explicit upper bounds for defining equations for the image of a polynomial map in the framework of noetherianity for polynomial functors.
\end{abstract}

\maketitle

\section{Introduction}

The study of decompositions of polynomials into simple summands is a central theme in algebraic geometry with wide-ranging applications. In this work, we study complex algebraic varieties of homogeneous polynomials admitting additive decompositions in terms of reducible polynomials, called \textit{strength decompositions} and, in the special case of linear factors, \textit{slice rank decompositions}. The geometry of these decompositions plays a role in the study of complete intersections contained in hypersurfaces \cite{CCG08}, and is central in the resolution of Stillman's conjecture \cite{AnHoc:StillmanConj} and the study of singular loci \cite{KZ18}. Moreover, the geometry of slice rank is essentially equivalent to the geometry of the Fano scheme of hypersurfaces, an object classically studied in algebraic geometry \cite{DebMan,DerkEggSnow}. In infinite dimensional algebraic geometry, strength plays a crucial role: in a way, it is a universal measure for the expressive power of polynomials and tensors \cite{BDK19,BDDE22}. In algebraic complexity theory, a restricted version of strength is used in \cite{GGIL22} as a coarsening of the algebraic branching program width of polynomials, and provided new methods for lower bounds using intersection theory and Noether-Lefschetz theory.

The \emph{strength} of a homogeneous polynomial $f$ is the smallest $r$ for which there is an expression
\[
f = g_1 h_1 + \cdots + g_rh_r
\]
of $f$ as sum of reducible homogeneous polynomials; here $g_i,h_i$ are homogeneous of degree strictly smaller than $\deg (f)$. The \emph{slice rank} of $f$ is the smallest $r$ for which there is an expression
\[
f = \ell_1 h_1 + \cdots + \ell_r h_r
\]
where $\ell_i$ are linear forms and $h_i$ are homogeneous polynomials of degree $\deg(f) - 1$.

Analogously to other additive decompositions of polynomials and tensors, such as Waring rank or tensor rank decompositions, the notions of strength and slice rank are controlled geometrically by membership into corresponding algebraic varieties: the \textit{secant varieties of the varieties of reducible forms}, and \textit{of forms having a linear factor}, respectively. See \Cref{subsec: slice rank} for the precise definitions. Determining complete sets of defining equations for secant varieties, and for these varieties in particular, is considered a hard problem. In this work, we provide some defining equations for higher strength and slice rank and prove a reduction result for the search space of set-theoretic equations in the case of cubic forms of slice rank two. More precisely, we have three main contributions:
\begin{itemize}[leftmargin=\parindent]
 \item The set of polynomials of strength one, that is the reducible polynomials, is an algebraic variety. In the case of polynomials in three variables, set-theoretic equations for this variety were determined in  \cite{Rup:ReduziabilitatKurven} in terms of the existence of special sections of the cotangent bundle of $\bbP^2$. We review this result, and we reinterpret the resulting determinantal equations in the language of Young flattenings. We provide several representation-theoretic insights, and prove that Ruppert's equations yield set-theoretic equations for the variety of reducible forms in any degree and any number of variables. See \Cref{thm: ruppert original} and \Cref{thm: ruppert any vars}.
\item We introduce new equations for the (closure of the) set of polynomials of strength at most $r$. These are determinantal equations, constructed as a generalization of the ones of \cite{Rup:ReduziabilitatKurven}. They encode the existence of special syzygies among the partial derivatives of a polynomial $f$ of small strength. Prior to this work, the only known equations for small strength were built on the non-emptiness of the singular locus of the hypersurface $\{ f= 0\}$, see \Cref{prop: singular locus}. The new equations enrich and refine this point of view. We prove more refined results in the case of cubic forms of strength two in five variables; see \Cref{thm: 11 syzygies for slrk 2}.
 \item We prove an inheritance result for the variety of cubic forms and slice rank at most $2$. In geometric terms, we prove that a cubic hypersurface $X = \{ f= 0\}$ contains a linear space of codimension two if and only if a generic linear sections $X \cap \bbP^{13}$ does. This implies that set-theoretic equations for 
 the second secant variety of the variety of forms having a linear factor are obtained by pulling back the equations of cubics of slice rank two in (at most) fourteen variables. See \Cref{thm:restrictions_vars}.
\end{itemize}

\Cref{thm:restrictions_vars} should be placed into the framework of $\bfVec$-varieties into polynomial functors. We refer to \Cref{subsec: functors} for the definitions; we point out here that the variety of forms of degree $d$ having a factor of degree $k$, as well as its secant varieties, can be realized as $\bfVec$-varieties, in the sense that associating to a vector space $V$ such subvariety of $S^dV$ is a \emph{functor} from the category of vector spaces to the category of varieties. In this setting, \cite[Corollary~3]{Draisma:TopologicalNoetherianityPolyFunctors} guarantees that ``set-theoretic equations for these varieties are determined in finite dimension'': for instance, for every $r \in\bbN$, there exists a value $n_0$ depending on $r$ such that, for every $n$, a homogeneous polynomial $f$ of degree $d$ in $n$ variables has slice rank at most $r$ if and only if all its restrictions to $n_0$ variables have slice rank at most $r$. In \cite{BDV24}, this finiteness result was made algorithmic, providing a finite theoretical procedure that determines a vector space $U_0$ of dimension $n_0$ from which one pulls back set-theoretic equations for the image of a polynomial map. The variety of cubic forms of slice rank at most two was an important guiding example of a closed subset of a polynomial functor, where even an upper bound on $n_0$ was unknown, see \cite[Example~1.4.1]{BDV24}.

Determining the integer $n_0$ for a given $\bfVec$-variety, or even providing upper bounds, is challenging. It has been achieved only in very few cases, and often for trivial reasons. For instance, a conciseness argument allows one to say that a system of set-theoretic equations for the $r$-th secant variety of Segre varieties, Veronese varieties and Segre-Veronese varieties is determined in dimension $n_0 = r+1$ \cite[Corollary~7.4.2.3]{Lan12}. The same argument shows the bound $n_0 \leq d+1$ for the Chow variety of completely reducible forms of degree $d$ and the bound $n_0 \leq dr+1$ for its $r$-th secant variety \cite[Remark~8.6.2.5]{Lan12}. In fact, in these cases the statement is true also scheme-theoretically and ideal-theoretically by \cite[Proposition~7.1.2(b)]{Wey:VB}.

\Cref{thm: ruppert any vars} guarantees that set-theoretic equations for the variety of reducible forms (and in fact, for any of its components) are determined in dimension $n_0 = 3$. The same argument shows an analogous result for varieties of polynomials with factors of specified degree; for instance, the bound $n_0 \leq d+1$ for the Chow variety can be upgraded to $n_0 = 3$, which is indeed attained by a classical system of equations, known as Brill's equations \cite{Gua18}. A similar argument can be used to see that the variety of polynomials of degree $e\cdot d$ of the form $g^d$ with $g$ homogeneous of degree $e$ has set-theoretic equations determined by restrictions to binary forms, that is $n_0 = 2$; in fact, these equations define the coincident root loci \cite{HilbPowers,Ch04,AbCh07}. A slightly more involved example is given in \cite[Proposition~3.1]{ChrGesZui:GapSubrank}: the variety of tensors of partition rank one, a tensor analog of the strength studied in the context of additive combinatorics \cite{Nas20}, has set-theoretic equations determined by restrictions to $\bbC^2 \ootimes \bbC^2$. 

\Cref{thm:restrictions_vars} guarantees that the variety of cubics of slice rank two has set-theoretic equations determined by restrictions to spaces of dimension fourteen; in particular, in this case $n_0 \leq 14$. To the best of our knowledge, this is the first example where an explicit upper bound is given without relying on straightforward conciseness considerations or simple genericity conditions. We propose a conjectural optimal value for $n_0$ in general in \Cref{conj: general restrictions}.

Throughout the paper, some claims are verified via direct computations using computer algebra software; we indicate when this is the case. For these computations, we used a Lenovo ThinkBook 14 G2 ITL, Intel Core i5 processor at 2.4 GHz with 8GB RAM, running Debian 11. We used the computer algebra system Macaulay2 \cite{M2}, v.1.21. We ran numerical experiments using Julia, v.1.11, and the package HomotopyContinuation.jl, v.2.13 \cite{BreTim}. Since our computations employ standard algorithms and routines, we do not provide additional code.

\section{Preliminaries}\label{sec: preliminaries}
Throughout the paper $V$ denotes a $(n+1)$-dimensional vector space over $\bbC$. Let $\{x_0,\ldots,x_n\}$ be a basis of $V$, and let 
\[
\Sym V \simeq  \bbC[x_0,\ldots,x_n]
\]
be the symmetric algebra of $V$, identified with the ring of polynomials on $V^*$. The subspace of homogeneous polynomials of degree $d$ is denoted by $S^{d}V$.

A \emph{variety} is an affine or a projective algebraic variety, possibly reducible. For a subset $X \subseteq \bbP^N$, write $I(X)$ for the ideal of polynomial equations in $N+1$ variables vanishing on $X$. For a subset $G$ of homogeneous polynomials, let $Z(G)$ be the variety of points of $\bbP^N$ defined by the vanishing of the elements of $G$. In particular, given $f\in S^{d}V$, $Z(f) \subseteq \bbP V^*$ denotes the hypersurface defined by the vanishing of $f$. 

For a variety $X\subseteq \bbP^N$, let $\sigma_r(X)$ denote the $r$-th \emph{secant variety} of $X$, that is the closure of the set of points lying on $r$-secant planes to $X$:
\[
\sigma_r(X) = \bar{\bigcup_{p_1 \vvirg p_r \in X} \langle p_1 \vvirg p_r\rangle};
\]
the overline denotes the closure, equivalently in the Zariski or the Euclidean topology.

\subsection{Strength and slice rank}\label{subsec: slice rank}
Let $f\in S^{d}V$ be a homogeneous polynomial. The \emph{strength} of $f$ is
\[
\str(f) \coloneqq \min  \biggl\{ r\in\bbN  :  f = \textstyle \sum_{i=1}^r g_i h_i \text{ for some } g_i\in S^{k}V,\, h_i\in S^{d-k}V \text{ with } k<d \biggr\}.
\]
The \emph{slice rank} of $f$ is
\[
\slrk(f) \coloneqq \min  \biggl\{ r\in\bbN  :  f = \textstyle \sum_{i=1}^r  \ell_i h_i \text{ for some } \ell_i \in S^{1}V,\, h_i\in S^{d-1}V \biggr\}.
\]
For any $f\in S^{d}V$, it is clear that $\str(f) \leq \slrk(f)$. It is immediate that, when $d =3$, the notions of strength and slice rank coincide. Moreover, strength and slice rank are equal for generic forms in any number of variables and any degree, see \cite[Theorem~1.8]{BBOV23}, but there are explicit examples where their gap can be arbitrarily large, see \cite[Proposition 3.2]{BikOne:StrengthGeneralPoly}. 

For every $k = 1 \vvirg d-1$, let 
\[
\calR^k_{n,d} \coloneqq \Set{ [f] \in \bbP S^{d} V : f =gh \text{ for some } g\in S^k V,\,h\in S^{d-k}V};
\]
note that $\calR^k_{n,d} = \calR^{d-k}_{n,d}$. The union 
\[
\calR^\bullet_{n,d} \coloneqq \bigcup_k \calR^k_{n,d}
\]
is the {\it variety of reducible forms}. In particular, for any $f\in S^{d}V$, we have $\slrk(f) = 1$ if and only if $[f] \in \calR^1_{n,d}$ and $\str(f) = 1$ if and only if $[f] \in \calR^{\bullet}_{n,d}$.

The secant varieties $\sigma_r(\calR^1_{n,d})$ and $\sigma_r(\calR^\bullet_{n,d})$ are the main object of study of our work. Since $\smash{\calR^\bullet_{n,d}}$ is reducible, its secant varieties are also reducible. By \cite[Theorem~1.8]{BikOne:StrengthGeneralPoly}, we have that $\sigma_r(\calR^1_{n,d})$ is an irreducible component of $\sigma_r(\calR^\bullet_{n,d})$. 

A geometric characterization of the slice rank of a homogeneous polynomial $f$ can be given in terms of the existence of linear subspaces contained in the hypersurface $Z(f)$. One important consequence of this characterization is that slice rank is lower semicontinuous: in particular, in the definition of $\sigma_r(\calR^1_{n,d})$ as (closure of the) union of all $r$-secant planes to $\smash{\calR^1_{n,d}}$, the closure operation is redundant. This result is proved for cubic forms in \cite{DerkEggSnow}, but the same proof applies in general. 
\begin{proposition}[{\cite[Proposition~2.2]{DerkEggSnow}}]
\label{prop: slice equivalent to fano}
Let $f \in S^{d} V$. Then 
\[
\slrk(f) = \min \bigl\{ r\in\bbN : \text{there is a linear space $L\subset Z(f)$ with $\codim L=r$} \bigr\}. 
\]
In particular, $\slrk(f)$ is lower semicontinuous.
\end{proposition}
In contrast to the result for the slice rank of \Cref{prop: slice equivalent to fano}, the strength is not, in general, lower semicontinuous; see \cite[Theorem~1.3]{BBOV:StrengthNotClosed}. One defines the \emph{border strength} to be the semicontinuous closure of the strength; more precisely,
\[
\ustr(f) \coloneqq \min\Set{r\in\bbN : f \in \sigma_r( \calR^\bullet_{n,d}) }.
\]

The dimension of $\sigma_r(\calR_{n,d}^1)$ is not difficult to compute. One possible proof is built on a vector bundle construction using \Cref{prop: slice equivalent to fano} and is given in \cite{Man:HypersurfacesEspacesLineaires}. 
\begin{proposition}[{\cite[Theorem, p.~308]{Man:HypersurfacesEspacesLineaires}}]
\label{prop: dimension secants slice}
For every $n,d\in\bbN$, with $d \geq 3$, 
\[
\codim \sigma_r(\calR^1_{n,d}) = \max\left\{ 0 , \binom{d+n-r}{n-r} - r(n-r+1) \right\}.
\]
In particular, 
\[
\codim\calR_{n,d}^1=\binom{d+n-1}{n-1}-n.
\] 
\end{proposition}
Determining the dimension of the other irreducible components of $\sigma_r(\calR_{n,d}^\bullet)$ is challenging, and it is related to deep open problems in commutative algebra, such as Fr\"oberg's conjecture \cite[p.~120]{Fro:Ineq}. We refer to \cite{catalisano2019secant,BikOne:StrengthGeneralPoly} for a related discussion.

To the best of our knowledge, the only effective method for determining lower bounds on strength and slice rank relies on the fact that whenever $f$ has low strength, then the hypersurface $Z(f)$ is singular. More precisely, we have the following result.
\begin{proposition}\label{prop: singular locus}
 Let $f \in S^{d} V$. If $\ustr(f) \leq r$, then $\codim \bigl(\Sing(f)\bigr) \leq 2r$. In particular, if $f$ is smooth, then
 \[
 \ustr(f) \geq \biggl\lceil \frac{n+1}{2} \biggr\rceil.
 \]
\end{proposition}
The content of \Cref{prop: singular locus} is mentioned in \cite[Remark~1.1]{AnHoc:StillmanConj}, and is presented more explicitly in \cite[Remark~4.3]{BBOV:StrengthNotClosed}. A complete proof is given in \cite[Proposition~6]{GGIL22}.
By construction, the best lower bound that \Cref{prop: singular locus} can provide for the strength of a polynomial in $n+1$ variables is $\lceil \frac{n+1}{2} \rceil$, whereas if $f$ is generic, by \Cref{prop: dimension secants slice} and the results of \cite{BBOV23}, we have 
\[
\str(f) = n - o(n).
\] 
A small improvement for slice rank lower bounds was given in \cite[Theorem~18]{GGIL22}, providing explicit examples of polynomials of slice rank $\frac{n+1}{2} +1$ for odd $n$, see \cite[Lemma~19]{GGIL22}.
\begin{proposition}[{\cite[Theorem~18]{GGIL22}}]\label{prop: singular locus compt}
Let $f \in S^{d} V$ be a form such that \[\codim \bigl(\Sing(f)\bigr) = 2r.\] Then $\slrk(f) = r$ if and only if there is a linear space $L \subseteq Z(f)$ with $\codim L = r$ that contains one of the irreducible components of $\Sing(f)$.
\end{proposition}

The condition described in \Cref{prop: singular locus} can be translated into equations for the variety $\sigma_r(\calR^{\bullet}_{n,d})$, namely explicit elements of $I( \sigma_r(\calR^{\bullet}_{n,d}) ) \subseteq \bbC[S^d V]$. There is usually little value in doing this explicitly because often the geometric condition that $Z(f)$ is singular in codimension $2r$ is easier to check rather than evaluating the equations that arise from such conditions. In this paper, we are not interested in realizing the equations for $\sigma_r(\calR^{\bullet}_{n,d})$ as explicit polynomials on $\bbC[S^d V]$; however, we will provide some information on their degree, and the complexity of the explicit evaluation, whenever this is possible. We refer to \cite[Section~2]{vdBDGIL:Metapolynomials} for a discussion on \emph{metacomplexity}, the area of algebraic complexity theory concerning complexity properties of equations of varieties which themselves control a complexity measure; in fact,  $\sigma_r(\calR^{\bullet}_{n,d})$ is an example of such.

\subsection{Polynomial functors and \textbf{Vec}-varieties}\label{subsec: functors}
We briefly introduce the framework of polynomial functors and polynomial maps between them. For more details, we refer to \cite{Draisma:TopologicalNoetherianityPolyFunctors}. Let $\bfVec$ be the category of finite-dimensional complex vector spaces and let $\bfVar$ be the category of complex affine algebraic varieties. A polynomial functor $P \colon \bfVec \to \bfVec$ is a direct sum of Schur functors, in the sense of \cite[Lecture 6]{FultonHarris}. For instance, the $d$-th symmetric power $S^d(-)$ is a polynomial functor which associates to a vector space $V$ the homogeneous component of degree $d$ in its symmetric algebra. A $\bfVec$-variety $X(-) \subseteq P(-)$ is a functor $X\colon \bfVec \to \bfVar$ with the property that, for every vector space $V$, the variety $X(V)$ is an affine variety of $P(V)$. For instance, $\smash{\calR^1_{-,d}} \subset S^d(-)$ assigning to a vector space $V$ the variety $\smash{\calR^1_{\dim V-1,d}} \subseteq S^d V$ of forms of slice rank one is a $\bfVec$-variety. Similarly, $\calR^\bullet_{-,d}$, as well as the secant varieties $\sigma_r(\calR^1_{-,d})$ and $\sigma_r(\calR^1_{-,d})$, are $\bfVec$-varieties. 

The functorial property guarantees that if $X(-)$ is a $\bfVec$-variety, then, every linear map $f : V \to W$ between vector spaces induces a map $P(f) \colon X(V) \to X(W)$, given by the restriction of the linear map $P(f)\colon P(V) \to P(W)$. In the setting of this work, this simply reflects the fact that the varieties of interest are invariant under change of coordinates. As a consequence, the ideals $I\bigl(X(V)\bigr) \subseteq \bbC[P(V)]$ are $\GL(V)$-representations and linear maps $f: V \to W$ induce pullback maps $\bbC[P(W)] \to \bbC[P(V)]$ mapping $I\bigl(X(W)\bigr)$ into $I\bigl(X(V)\bigr)$.

In this setting, one may fix inclusions $\bbC^{n} \subseteq \bbC^{n+1}$, endowing $\bbC^{\infty} = \bigcup_{n \in \bbN} \bbC^n$ with the structure of a direct system. A $\bfVec$-variety is then a subvariety $X_\infty \subseteq P(\bbC^{\infty})$ acted on by $\GL_\infty =  \bigcup_{n\in \bbN} \GL_n$. In this way, the results of \cite{Draisma:TopologicalNoetherianityPolyFunctors} apply and one has the following fundamental result.
\begin{theorem}[{\cite[Theorem~1 and Corollary~3]{Draisma:TopologicalNoetherianityPolyFunctors}}]\label{thm: topological noeth}
  Let $X(-) \subseteq P(-)$ be a $\bfVec$-variety. Then $X(-)$ is set-theoretically cut out by finitely many $\GL_\infty$-orbits of polynomial equations. In particular, there exists an integer $n_0$ such that the $\GL(V)$-orbits of the equations in $I(X(\bbC^{n_0})) \subseteq \bbC[ P(\bbC^{n_0})]$ cut out $X(V) \subseteq P(V)$ for every $V$.
\end{theorem}
Following \Cref{thm: topological noeth}, for a given $\bfVec$-variety $X(-) \subseteq P(-)$, it is a natural problem to determine the value $n_0$ with the property that set-theoretic equations for $X(\bbC^{n_0})$ give set-theoretic equations for every $X(V)$.

We illustrate this phenomenon in a simple example. Let $S^2(-)$ be the polynomial functor given by the symmetric square and let $X_r(-) \subseteq S^2(-)$ be the $\bfVec$-variety defined by
\[
X_r(V) = \{ q \in S^2V : \rank(q) \leq r \},
\]
that is the variety of symmetric matrices of rank at most $r$. For every $V$, the variety $X_r(V)$ is cut out by the minors of size $r+1$ of the symmetric matrix, which are elements of $S^{r+1} S^2 V^*$. If $\dim V \leq r$, this system of minors is empty and $X_r(V) = S^2V$. If $\dim V \geq r+1$, it is generated by the $\GL(V)$-orbits of a single principal minor, for instance the determinant of the top-left $(r+1) \times (r+1)$ submatrix. In particular, a set of generators for the system of minors arises via pullback of the single equations of $X(\bbC^{r+1})$. Therefore, in this case, $n_0 = r+1$.

In the setting of slice rank and strength, \Cref{thm: topological noeth} has the consequence that, for every $d,r$, there exists $n_0 = n_0(d,r)$ with the property that, for every $n$, a homogeneous polynomial $f \in S^d V$ belongs to $\sigma_r(\calR^\bullet_{n,d})$ if and only if all its restrictions to $n_0$ variables belong to $\sigma_r(\calR^\bullet_{n,d})$. A straightforward argument based on Bertini's Theorem allows us to see that, if $r = 1$, one can reduce to forms in $n_0 + 1 = 3$ variables; see \Cref{thm: ruppert any vars}. In \Cref{thm:restrictions_vars}, we will prove the upper bound $n_0 +1 \leq 14$ in the case $(d,r) = (3,2)$. 

\section{Equations for forms of strength one}
\label{sec:ruppert}

A complete set of set-theoretic equations for the variety $\calR^\bullet_{2,d}$ of reducible forms in three variables was determined by W.~Ruppert in \cite{Rup:ReduziabilitatKurven}. We refer to them as \emph{Ruppert's equations}. They arise as a system of determinantal equations of degree $d^2-1$. In this section, we review this result providing some geometric and representation-theoretic insights. We then prove an extension of this result to any number of variables, and we conclude with a discussion on equations for cubics of slice rank one arising from the structure of their isotropy groups.

\subsection{Set-theoretic equations for reducible ternary forms}
Let $V$ be a vector space of dimension $n+1$ with basis $x_0 \vvirg x_n$. Let $\partial_0 \vvirg \partial_n$ be the basis of $V^*$ dual to $x_0 \vvirg x_n$. The space $V^*$ acts naturally on the symmetric algebra $\Sym V$ by tensor contraction. In coordinates, the action is given by differentiation: for $g \in \Sym V$,
\[
\partial_i \cdot g= \frac{\partial}{\partial x_i} g.
\]
For every $e \geq 1$, consider the subspace of $V^* \otimes S^e V$ defined by 
\[
\fraksl^{(e-1)}(V) = \Bigl\{ \textstyle \sum_{i=0}^n \partial_i \otimes g_i  : \textstyle \sum _{i=0}^n \partial_i g_i = 0 \Bigr\} \subseteq V^* \otimes S^e V,
\]
that is called the $(e-1)$-th \emph{prolongation} of the Lie algebra $\fraksl(V)$ \cite{SiuHo73}. It naturally fits into the exact sequence 
\begin{equation}\label{eq: euler sequence via pieri}
0 \to \fraksl^{(e-1)}(V) \to V^* \otimes S^{e}V \to S^{e-1}V \to 0,
\end{equation}
where the second map is given by contraction. Note that when $e = 1$, the contraction map is the the trace map $V^* \otimes V = \End(V) \to \bbC$ and $\fraksl^{(0)}(V)$ coincides with $\fraksl(V)$, the Lie algebra of traceless endomorphisms. The exactness of \eqref{eq: euler sequence via pieri} implies 
\begin{equation}\label{eqn: dim slpV}
\dim \fraksl^{(e-1)}(V) = \dim (V^* \otimes S^{e} V) - \dim S^{e-1}V =  \frac{n(n+1+e)}{e} \binom{n+e-1}{n} . 
\end{equation}
Prolongation can be defined for every submodule of a polynomial ring and, even more generally, for every linear series of sections of a vector bundle. It appears in the study of exterior differential systems \cite[Chapter~8]{IvLan:Cartan}, of ideals of secant varieties \cite{SidSul:ProlontationCompAlgebra} and of Lie algebras of isotropy groups  \cite{SiuHo73}.

From a representation theoretic point of view, using Pieri's rule, one can see that the space $V^* \otimes S^e V$ decomposes, as a $\GL(V)$-representation, into two summands:
\[
V^* \otimes S^e V = S^{e-1}V \oplus \fraksl^{(e-1)}(V).
\]
The summand $\fraksl^{(e-1)}(V)$ is isomorphic, as an $\SL(V)$-representation, to the Schur module $\bbS_{(e+1,1^{n-1})}V$. In particular \eqref{eq: euler sequence via pieri} is one of the two exact sequences induced by this direct sum decomposition, and there is an analogous sequence
\begin{equation}\label{eq: euler sequence via pieri 2}
0 \to S^{e-1}V \to S^{e}V \otimes V^* \to \fraksl^{(e-1)}(V) \to  0,
\end{equation}
where the embedding $S^{e-1} V \to V^* \otimes S^eV$ is given by identifying $g \in S^{e-1} V$ with the symmetrization of $g \otimes \id_V \in S^{e-1} V \otimes V \otimes V^*$.

For every $e_1,e_2 \geq 0$, let $\mu \colon S^{e_1} V \otimes S^{e_2} V \to S^{e_1 + e_2} V$ be the multiplication map, defined on decomposable tensors by $\mu(f_1 \otimes f_2) = f_1f_2$ for every $f_1\in S^{e_1}V$ and $f_2\in S^{e_2}V$, and extended linearly.

\begin{definition}[Ruppert map]
For $f \in S^{d} V$, $e \geq 1$, the \textit{$e$-th Ruppert map} is
\begin{equation}\label{formula:Ruppert_map}
\begin{tikzcd}[row sep=0pt,column sep=1pc]
 \rho_f^{(e)}\colon \fraksl^{(e-1)} V\arrow{r} & S^{e+d-1} V\hphantom{.} \\
  {\hphantom{\rho_f^{(e)}\colon{}}}  {\displaystyle \sum_{i=0}^n g_i \otimes \partial_i} \arrow[mapsto]{r} & {\displaystyle \sum_{i=0}^n g_i \partial_i f},
\end{tikzcd}
\end{equation}
defined as the composition of the contraction map $\fraksl^{(e-1)} V \to S^{e}V \otimes S^{d-1} V$ defined by contracting $f$ against the factor $V^*$ of $\fraksl^{(e-1)} V \subseteq S^{e} V \otimes V^*$, and the multiplication map $S^{e}V \otimes S^{d-1} V \to S^{e+d-1}V$.
\end{definition}
The $e$-th Ruppert map is therefore an example of Young flattening, which were introduced in \cite[Section~4]{LanOtt}, see also \cite[Section~7.8]{Lan12}; it is, in fact, the restriction of the first shifted partial derivatives map from \cite{GKKS} to the prolongation $\fraksl^{(e-1)} V$.

The main result of \cite{Rup:ReduziabilitatKurven} shows that rank conditions on the map $\rho_f^{(e)}$ provide set-theoretic equations for the variety $\calR^\bullet_{2,d}$ of ternary forms of strength one. We refer to \cite[Section~3.2]{Schin:PolySpecialRegardRed} for a version of the same proof in English.

\begin{theorem}[{\cite{Rup:ReduziabilitatKurven}}]
\label{thm: ruppert original}
 Let $\dim V = 3$ and $f \in S^{d} V$. The following are equivalent: 
 \begin{enumerate}[label=(\roman*)]
     \item $f$ is reducible, that is, $f \in \calR^\bullet_{2,d}$;
     \item $\rank (\rho_f^{(d-2)}) < \dim \fraksl^{(d-3)}(V) = d^2 - 1$.
 \end{enumerate}
In particular, the ideal $I_{d^2-1}(\rho_f^{(d-2)})$ generated by the minors of size $d^2 - 1$ of $\rho_f^{(d-2)}$ cuts out $\calR^\bullet_{2,d}$ set-theoretically.
\end{theorem}
We give a proof of one of the two implications, which serves as a preparatory result for the rank conditions described in \Cref{sec: more equations} for higher strength. 
\begin{proposition}\label{prop: ruppert_first_implication_new}
Let $\dim V = 3$ and $f \in S^{d} V$. If $f$ is reducible, then
\[
\rank (\rho_f^{(d-2)}) < \dim \fraksl^{(d-3)}(V) = d^2 - 1.
\]
\end{proposition}
\begin{proof}
The condition $\rank (\rho_f^{(d-2)}) < \dim \fraksl^{(d-3)}(V)$ is equivalent to the non-injectivity of the Ruppert map $\rho_f^{(d-2)}$. Therefore, by semicontinuity of matrix rank, it suffices to determine an element in $\smash{\ker( \rho_f^{(d-2)} )}$ for a generic $f \in \smash{\calR_{n,d}^k}$ and for every $k$. Let $g \in S^k V$ and $h \in S^{d-k} V$ be generic polynomials and let $f = gh$. Define
\[
\Delta \coloneqq \det \begin{pNiceMatrix} 
\partial_0g & \partial_1g & \partial_2g \\ 
\partial_0h & \partial_1h & \partial_2h \\ 
\partial_0  & \partial_1  & \partial_2  
\end{pNiceMatrix} \in S^{d-2} V \otimes V^*,
\]
which is to be read as the result of the Laplace expansion of the determinant, with the multiplication between elements of the last row and elements of the other rows identified with the tensor product. Explicitly, setting 
\[
\Delta_0\coloneqq \det \begin{pNiceMatrix} 
\partial_1g & \partial_2g \\ \partial_1h & \partial_2h 
\end{pNiceMatrix}, \qquad 
\Delta_1 \coloneqq -\det \begin{pNiceMatrix} 
\partial_0g & \partial_2g \\ \partial_0h & \partial_2h 
\end{pNiceMatrix}, \qquad 
\Delta_2 \coloneqq \det \begin{pNiceMatrix} 
\partial_0g & \partial_1g \\ \partial_0h & \partial_1h 
\end{pNiceMatrix},
\]
one has $\Delta = \Delta_0 \otimes \partial_0 + \Delta_1 \otimes \partial_1 + \Delta_2 \otimes \partial_2$.

We show that $\Delta$ is a non-trivial element of $\smash{\ker( \rho_f^{d-2} )}$. First, note that $\Delta \neq 0 $, because $g,h$ are chosen generically. Then, we have
\begin{align*}
\sum_{i=0}^2\partial_i \Delta_i =  
                         \partial_0 (\partial_1 g \partial_2 h - \partial_2g \partial_1h) + 
                         \partial_1 (\partial_2 g \partial_0 h - \partial_0g \partial_2h) + \partial_2 (\partial_0 g \partial_1 h - \partial_1g \partial_0h)=0,
\end{align*}
so $\Delta \in \fraksl^{(d-3)}(V)$. Finally, by Leibniz's rule and the linearity of the determinant in the last row, we have
\begin{align*}
 \rho_f^{(d-2)} (\Delta ) = \Delta(f) &= \det \begin{pNiceMatrix} 
\partial_0g & \partial_1g & \partial_2g \\ 
\partial_0h & \partial_1h & \partial_2h \\ 
\partial_0f  & \partial_1 f & \partial_2f
\end{pNiceMatrix}\\[1ex]
&= \det \begin{pNiceMatrix} 
\partial_0g & \partial_1g & \partial_2g \\ 
\partial_0h & \partial_1h & \partial_2h \\ 
h\partial_0 g  & h\partial_1g  & h\partial_2 g 
\end{pNiceMatrix} + \det \begin{pNiceMatrix} 
\partial_0g & \partial_1g & \partial_2g \\ 
\partial_0h & \partial_1h & \partial_2h \\ 
g\partial_0h  & g\partial_1h  & g\partial_2h  
\end{pNiceMatrix} = 0 ;
\end{align*}
therefore $\Delta \in \ker (\rho_f^{(d-2)})$. 
\end{proof}
The proof of \Cref{thm: ruppert original} in \cite{Rup:ReduziabilitatKurven} and \cite{Schin:PolySpecialRegardRed} is given in the language of differential forms. The condition that is being proved is that $f$ is reducible if and only if there is a closed meromorphic differential form $\omega$ on $\bbP^2$ with poles along the curve $Z(f) \subseteq \bbP^2$. In this language, the element $\Delta$ introduced in the proof of \Cref{prop: ruppert_first_implication_new} is essentially the same as the desired differential form. More precisely, define
\[
\omega_f = \frac{1}{f} [(x_1 \Delta_2 - x_2 \Delta_1) \diff x_0 + (x_2 \Delta_0 - x_0 \Delta_2) \diff x_1  + (x_0 \Delta_1 - x_1 \Delta_0) \diff x_2 )].
\]
By definition $\omega_f$ has poles only along $\smash{Z(f)}$ and one can verify it is closed. In fact, the condition that $\omega_f$ is closed is equivalent the fact that $\smash{\Delta \in \ker (\rho_f^{(d-2)})}$. If $f = gh$ with $g,h$ distinct, then, up to scaling, $\omega_f = \smash{\diff \log (g^{\deg (h)} / h^{\deg(g)})}$, which can be verified by expanding the derivatives and using Euler's formula for homogeneous polynomials. With this formulation, closedness is immediate. 

The condition that $\omega_f$ is a well-defined differential form on $\bbP^2$ is equivalent to the condition that $\Delta \in \fraksl^{(d-3)}(V)$. The correspondence is more general. Consider the classical Euler sequence for the tangent bundle $T\bbP^2$ of $\bbP^2 = \bbP V^*$
\[
\begin{tikzcd}
0 \arrow{r} &\calO_{\bbP^2} \arrow{r} &V^* \otimes \calO_{\bbP^2}(1) \arrow{r} &T \bbP^2 \arrow{r} &0;
\end{tikzcd}
\]
twisting by $\calO_{\bbP^2}(d-1)$ and passing to global sections yields the exact sequence 
\[
\begin{tikzcd}
0 \arrow{r} & H^0(\calO_{\bbP^2}(d-1)) \arrow{r} & V^* \otimes H^0(\calO_{\bbP^2}(d)) \arrow{r} & H^0(T\bbP^2(d)) \arrow{r} & 0
\end{tikzcd}
\]
because $H^1(\calO_{\bbP^2}(d-1)) = 0$ as $d \geq 1$.  This sequence coincides with the sequence \eqref{eq: euler sequence via pieri 2} of vector spaces induced by the differentiation map; in particular 
\[
H^0\bigl(T \bbP^2 (d)\bigr) = \fraksl^{(d-1)}(V).
\]
Dually, the Euler sequence for the cotangent bundle $T^* \bbP^2$ 
\[
0 \to T^* \bbP^2 \to V \otimes \calO_{\bbP^2}(-1) \to \calO_{\bbP^2} \to 0
\]
defines, after twisting by $\calO_{\bbP^2}(d)$ and passing to global sections, the exact sequence
\[
0 \to H^0\bigl(T^* \bbP^2 (d)\bigr) \to V \otimes H^0\bigl(\calO_{\bbP^2}(d-1)\bigr) \to H^0\bigl(\calO_{\bbP^2} (d)\bigr) \to 0
\]
where the second map is the multiplication map $V \otimes S^{d-1} V \to S^{d} V$, showing that $\smash{H^0(T^* \bbP^2 (d))}$ is isomorphic to the Schur module $\smash{\bbS_{(d,1)}V}$. Under this duality one can see that $\Delta \in \fraksl^{(d-1)}(V)$ if and only if $\omega_f \in \bbS_{(d,1)}V$.  
In particular, \Cref{prop: ruppert_first_implication_new} can be regarded as a reinterpretation of the fact that if $f$ is reducible then there is a meromorphic closed differential form on $\bbP^2$ with poles along $Z(f)$. The reverse implication is not straightforward: the proof of \cite{Schin:PolySpecialRegardRed} is a direct calculation in local coordinates; a proof using the geometry of the sheaf of logarithmic differentials, or a related cohomological construction, would shed some light on the potential of similar methods for higher strength.

\subsection{Set-theoretic equations for reducible forms in higher number of variables}

We use \Cref{thm: ruppert original}, together with a Bertini type argument, to obtain set-theoretic equations for the variety $\calR^\bullet_{n,d}$ of reducible forms for every $n$.

\begin{theorem}\label{thm: ruppert any vars}
    Let $\dim V = n+1$ and $f \in S^{d} V$. The following statements are equivalent:
    \begin{enumerate}[label=(\roman*)]
    \item $f$ is reducible;
    \item for a generic subspace $E \subseteq V^*$ with $\dim E = 3$, the restriction $f|_E \in S^{d} E^*$ is reducible;
    \item for every subspace $E \subseteq V^*$ with $\dim E = 3$, the restriction $f|_E \in S^{d} E^*$ is reducible.
    \end{enumerate}
Moreover, there exist finitely many $E_1 \vvirg E_N \subseteq V^*$ linear subspaces with $\dim E_j = 3$, for every $j=1,\dots,N$, such that the ideal 
\[
I = I_{d^2-1}(\rho_{f|_{E_1}}^{(d-2)}) + \cdots +  I_{d^2-1}(\rho_{f|_{E_N}}^{(d-2)})
\] 
cuts out $\calR_{n,d}^\bullet$ set-theoretically.
\end{theorem}
\begin{proof}
$(i) \Rightarrow (iii)$. Let $f = gh$, with $g\in S^kV$ and $h\in S^{d-k}V$ for some $k \geq 1$. Let $E \subseteq V^*$ be a linear space with $\dim E = 3$. Then
\[
f|_E=g|_E\cdot h|_E,
\]
showing that $f|_E$ is reducible. 

The implication $(iii) \Rightarrow (ii)$ is clear.

$(ii) \Rightarrow (i)$.  
We show that if $f$ is irreducible and $E \subseteq V^*$, with $\dim E = 3$, is a generic linear space, then $f|_E$ is irreducible. This is a consequence of Bertini's Theorem \cite[Corollary 6.11(3)]{Jou:Bertini}, applied to the hypersurface $Z(f) \subseteq \bbP V^*$: such hypersurface is irreducible by assumption, hence a generic linear section is irreducible as well. 

The second part of the statement follows by noetherianity. For every choice of an element $E \in \Gr(3,V^*)$, the minors of size $d^2 - 1$ of $\smash{\rho^{d-2}_{f|_E}}$ are polynomials in the coefficients of the polynomial $f$. As $E$ varies in the Grassmannian $\Gr(3,V^*)$, the set of these equations defines $\calR^\bullet_{n,d}$ set-theoretically. By noetherianity, finitely many equations, that is, finitely many restrictions $E$, suffice.
\end{proof}

\Cref{thm: ruppert any vars} is a quantitative version of \Cref{thm: topological noeth}, as explained in \Cref{subsec: functors}. It shows that set-theoretic equations for the variety of reducible forms in $n+1$ variables arise as pullback of the equations of such variety for three variables; in the notation of \Cref{thm: topological noeth}, this shows that $n_0 = 3$ for the $\bfVec$-variety $\calR^\bullet_{-,d}$. The same argument shows that the equivalence of the three statements in \Cref{thm: ruppert any vars} holds for each component $\calR^k_{n,d}$ of $\calR^\bullet_{n,d}$: the fact that $f$ has a factor of degree $k$ can be checked on the restriction of $f$ to generic subspaces of dimension three.

A natural question concerns upper bounds on the number $N$ of restrictions required to obtain a system of set-theoretic equations for $\smash{\calR_{n,d}^\bullet}$ from the determinantal ideals of the restrictions $f|_{E_i}$ with $i = 1 \vvirg N$, or, even better, explicit choices of restrictions $E_1 \vvirg E_N$ yielding such equations. This problem is related to the study of algebraic matroids and identifiability in compressed sensing, see, e.g., \cite{LM24,GGU25}. We do not address it in this paper; we record some preliminary observations following from an immediate parameter count in the first non-trivial case.
\begin{remark}
When $(n,d)=(2,3)$, a direct calculation with the support of a computer algebra software, shows that the space of Ruppert's equations $I_{d^2-1}(\rho^{(d-3)}_f)$ in degree $8 = d^2-1$ is a copy of the module 
\[
\bbS_{(10,9,5)} \bbC^3 \subseteq S^8 S^3 \bbC^3;
\]
here $\bbS_{(10,9,5)} \bbC^3$ denotes a specific copy of the Schur module of weight $(10,9,5)$ in the \emph{plethysm} $S^8S^3 \bbC^3$, see, e.g., \cite[Lecture 11]{FultonHarris}.
We have 
\[
\dim \bbS_{(10,9,5)} \bbC^3 = 35;
\]
so this is a system of $35$ equations cutting out set-theoretically the subvariety $\calR^1_{2,3}$ of codimension $2$ in $\bbP S^3 \bbC^3 = \bbP^9$. 

Let $\dim V = n+1$ and consider the variety $\calR^1_{n,d} \subseteq \bbP S^3 V$. Every restriction to a generic subspace $E \subseteq V^*$ with $\dim E = 3$ contributes a system of $35$ equations to $\smash{I(\calR^1_{n,3})}$, and it cuts out set-theoretically a variety of codimension $2$, which is a cone over a corresponding $\calR^1_{2,3}$ lying in a subspace. Since, by \Cref{thm: ruppert any vars}, the base locus of all such equations is $\calR^1_{n,d}$, we expect $\lceil  (\dim \bbS_{(10,9,5)} \bbC^{n+1})/ 35 \rceil$ suffice to generate the whole module, and $\lceil (\dim S^3 \bbC^{n+1}) / 2 \rceil$ are sufficient to cut out $\calR^1_{n,d}$ set-theoretically. These claims are however hard to verify explicitly.
\end{remark}
For degree higher than three, the system $I_{d^2-1}(\rho^{(d-3)}_f)$ is not an irreducible $\GL(V)$-module. In principle, different components have, individually, base locus larger than the sole $\smash{\calR^\bullet_{n,d}}$ and therefore equations belonging to different irreducible modules could contribute differently to the system of equations for $\calR^\bullet_{n,d}$ in $S^d \bbC^{n+1}$. 

We conclude this section drawing some connections between Ruppert's equation and algebraic properties of the Jacobian ideal of a reducible polynomial $f$.
\begin{remark}\label{rmk: syz}
From the point of view of commutative algebra, we observe that Ruppert's equations from \Cref{thm: ruppert original}, and so those from \Cref{thm: ruppert any vars}, detect {\it unexpected syzygies} of the Jacobian ideal of $f$. From this perspective, Ruppert's equations are a refined version of the equations arising from \Cref{prop: singular locus}. More precisely, \Cref{prop: singular locus} simply states that if $f$ is reducible then $\Sing(Z(f))$ has codimension at most two. So, when $n\geq 2$, the Jacobian ideal $J = (\partial_0 f \vvirg \partial_n f)$ is not a complete intersection of codimension $n+1$. Therefore, it has extra syzygies, in addition to the Koszul syzygies. In fact, $\codim \Sing(Z(f)) \leq 2$ guarantees that every $3$-dimensional subspace $H \subseteq J_{d-1}$ of first order derivatives of $f$ has syzygies besides the Koszul relations. \Cref{thm: ruppert any vars} describes further properties of these additional syzygies: for every subspace $H \subseteq J_{d-1} \subseteq S^{d-1}V$ with $\dim H = 3$, the ideal generated by $H$ has at least one syzygy in degree $d-2$ and such a syzygy belongs to the subspace
\[
\fraksl^{(d-3)}(V) \subseteq S^{d-2}V \otimes V^*.
\]
Because of its degree, this must be independent from the Koszul syzygies, which are in degree $d-1$. The upshot of \cite{Rup:ReduziabilitatKurven} is that this necessary condition is also sufficient.
\end{remark}

\subsection{Isotropy groups and Ruppert's equations for cubic polynomials}

The isotropy group of a homogeneous polynomial $f \in S^d V$ is its stabilizer under the action of $\GL(V)$: 
\[
\Stab_{\GL(V)}(f) = \{ A \in \GL(V) : A \cdot f = f \}.
\]
Depending on the source, $\Stab_{\GL(V)}(f)$ is also called the \textit{symmetry group of $f$} \cite[Section~4.1.2]{Lan17}, or the \textit{linear preserver subgroup of $f$} \cite{LP01,GHL25}. By definition, the group $\Stab_{\GL(V)}(f)$ is a closed subgroup of $\GL(V)$; denote by $\Stab^\circ_{\GL(V)}(f)$ its \emph{identity component}, which is the unique connected (irreducible) component containing the identity and it is a normal subgroup of $\Stab_{\GL(V)}(f)$, see, e.g., \cite[Ch.I, Section~1.2]{Borel:LinAlgGp}. We use the notation $\Stab_{G}(f)$ to indicate the stabilizer of $f$ in the group $G$.

The Lie algebra of $\Stab_{\GL(V)}(f)$ is a subalgebra of $\frakgl(V) \simeq V^* \otimes V$ and its dimension coincides with 
\[
\dim \Stab_{\GL(V)}(f) = \dim \Stab^\circ_{\GL(V)}(f).
\]
For every $f \in S^d V$, one has a natural {\it orbit map}
\[
\begin{tikzcd}[row sep=0pt,column sep=1pc]
 \gamma\colon \GL(V)\arrow{r} & S^d V \\
  {\hphantom{\gamma\colon{}}}  A \arrow[mapsto]{r} & A \cdot f.
\end{tikzcd}
\]
The fiber of $\gamma$ over $f$ is, by definition, $\Stab_{\GL(V)}(f)$. The differential of $\gamma$ at the identity is given by the Lie algebra action
\[
\begin{tikzcd}[row sep=0pt,column sep=1pc]
\diff \gamma\colon \frakgl(V)\arrow{r} & S^d V \\
  {\hphantom{\diff \gamma\colon{}}}  X \arrow[mapsto]{r} & X. f,
\end{tikzcd}
\]
where we identify $S^d V$ with the tangent space $T_f S^d V$. The Lie algebra of $\Stab_{\GL(V)}(f)$ is the kernel of $\diff \gamma$. In particular, $\dim \Stab_{\GL(V)}(f) > 0$ if and only if $\diff \gamma$ is not injective. 

In the case $(d,n) = (3,2)$, the prolongation $\fraksl^{(d-3)}(V)$ coincides with the algebra $\fraksl(V)$ and the Ruppert map $\rho_f^{(1)} : \fraksl(V) \to S^3 V$ is given by the Lie algebra action on $f$. This yields the following consequence of \Cref{thm: ruppert original}.
\begin{corollary}\label{corol: stabilizer plane cubics}
Let $f \in S^3 \bbC^3$. Then $f$ is reducible if and only if the stabilizer of $f$ under the action of $\SL(V)$ has positive dimension.    
\end{corollary}
In fact, \Cref{corol: stabilizer plane cubics} can also be verified using the classification of plane cubics, for example following \cite{KogMaz}. We expand on this in \Cref{rmk: plane cubics}.

The following result uses a lower bound on the dimension of the isotropy group of reducible cubics to obtain equations for $\calR^1_{n,3}$. Note that since $d =3$, $\calR^\bullet_{n,3} = \calR^1_{n,3}$.
\begin{proposition}
\label{prop: stabilizer cubics}
  Let $V$ be a vector space of dimension $n+1$.  If $f \in S^3 V$ is reducible, then 
  \[
  \rank(\rho_f^{(1)}) \leq \frac{n(n+5)}{2}
  \] 
  and equality holds for generic $f$ in $\calR^1_{n,d}$. Moreover, if $n=2$, the converse holds as well.
\end{proposition}
\begin{proof}
Let $f\in S^3 V$ be a generic cubic in $\calR^1_{n,d}$. Then $f = \ell q $ for some $\ell\in S^1 V$ and $q\in S^2 V$. By genericity $q$ is a full-rank quadratic form. As observed before
\[
\rho_f^{(1)} : \fraksl(V) \to S^{d}V
\]
coincides with the Lie algebra action of $\fraksl(V)$ on $f$. Let $\frakg_f$ be the Lie algebra of $\Stab_{\GL(V)}(f)$. The desired statement is equivalent to the the fact that 
\[
\dim ( \frakg_f \cap \fraksl(V)) \geq (n+1)^2-1 -\frac{n(n+5)}{2} = \binom{n}{2}.
\]
By the discussion above, this is equivalent to the fact that $\Stab^\circ_{\SL(V)}(f)$ has dimension (at least) $\binom{n}{2}$. By unique factorization and the genericity of $\ell,q$, if $A \in \GL(V)$ satisfies $A \cdot f = f$ then $A \cdot \ell = \ell$ and $A \cdot q = q$, up to scaling. Let $\SO(q)$ be the orthogonal group stabilizing the quadric $q$. We deduce 
\[
\Stab^\circ_{\SL(V)}(f) = \Stab^\circ_{\SO(q)}(\ell).
\]
This is the stabilizer in the orthogonal group of a generic hyperplane: by construction, this is an orthogonal subgroup acting on such hyperplane, that is a copy of $\SO(n)$. We obtain $\dim \Stab^\circ_{\SL(V)}(f) =\binom{n}{2}$ and this concludes the proof of the first part.

The second part is a restatement of \Cref{corol: stabilizer plane cubics}.
\end{proof}
From \Cref{prop: stabilizer cubics}, we deduce that the minors of size $\frac{n(n+5)}{2} + 1$ of $\rho_f^{(1)}$ give equations for $\calR^1_{n,3}$. It is not immediately clear whether these equations belong to the ideal generated by the equations described in \Cref{thm: ruppert any vars}. We provide a partial converse of \Cref{prop: stabilizer cubics}.
\begin{proposition}\label{implication on reducible}
Let $V$ be a space with $\dim V = n+1$ and let $f\in S^3 V$. If $\Stab_{\SL(V)}^{\circ}(f)$ is isomorphic to $\SO(n)$, then $f$ is reducible. 
\end{proposition}
\begin{proof}
The isomorphism $\Stab^\circ_{\SL(V)}(f) \simeq \SO(n)$ defines an injective homomorphism $\varphi \colon \SO(n)\to \GL(V)$ which makes $V$ into a faithful $\SO(n)$-representation.

The only irreducible representations of $\SO(n)$ of dimension at most $n+1$ are either $1$-dimensional, with $\SO(n)$ acting trivially, or $n$-dimensional, with $\SO(n)$ acting in the standard way on $\bbC^n$. This can be proved as follows. By Weyl dimension formula \cite[Corollary~24.6]{FultonHarris}, every irreducible representation is larger, in dimension, than a corresponding fundamental representation. So it suffices to prove that the claim holds for fundamental representations. The fundamental weights for $\frakso(n)$ correspond to the exterior powers $\Lambda^k \bbC^n$ of the standard representations with $k \leq n/2$, or the spin representations. The exterior powers have dimension larger than $n+1$ except for $k=0,1$; the spin representations do not yield representations for the group $\SO(n)$ \cite[Chapter~18]{FultonHarris}. Therefore all irreducible representation except $\bbC^1 = \Lambda^0 \bbC^n$ and $\bbC^n = \Lambda^1\bbC^n$ have dimension at least $n+2$. A direct combinatorial argument by counting \emph{Gelfand-Tsetlin fillings} of suitable \emph{Bratteli diagrams} is also possible; we refer to \cite[Section~2.3 and Section~4.1]{BryGes} for an explanation.

Since $V$ is a faithful representation, we deduce $V = \bbC^1 \oplus \bbC^n$, with $\SO(n)$ acting trivially on $\bbC^1$ and in the standard way on $\bbC^n$. In particular, $\SO(n)$ is realized as the stabilizer of a quadric $q_0 \in S^2 \bbC^n \subseteq S^2 V$.

Explicitly, after possibly changing coordinates, we may assume that $q_0=x_1^2+\cdots+x_n^2$ so that
\[
\Stab^\circ_{\SL(V)}(f)=\begin{pmatrix}
1 & 0 \\
0 & \SO(n) \\
\end{pmatrix}\subseteq \SO(V)\subseteq \SL(V);
\]
in particular the quadric $q=x_0^2+q_0$ is stabilized by a copy of $\SO(q) \supseteq \SO(q_0)$. Write $\SO(V) = \SO(q)$.

Now, $\SO(V)$ acts on $S^3 V$, which decomposes into direct sum of irreducible modules
\begin{equation}
\label{eq:split}
S^3 V = \mathcal{H}_3 \oplus q\mathcal{H}_1=\mathcal{H}_3 \oplus qV, 
\end{equation}
where $\mathcal{H}_3$ is the space of degree three harmonic homogeneous polynomials and $\mathcal{H}_1 = V$, see \cite[Corollary 5.6.12]{GW09} or \cite[Proposition 3.11]{Fla24} for a direct proof.
Therefore, we have 
\[
f = h_3 + q h_1,
\]
for certain unique $h_1\in V$ and $h_3\in \mathcal{H}_3$. By uniqueness of the decomposition, both $h_3$ and $qh_1$ must be stabilized by $\smash{\Stab_{\SL(V)}^\circ(f)}$. The only linear form which is stabilized by $\Stab_{\SL(V)}^\circ(f) = \SO(n)$ is a multiple of $x_0$. Therefore, we may assume that $h_1= \lambda x_0$. Moreover, since $h_3$ is also stabilized by $\Stab^\circ_{\SL(V)}(f)$, which is contained in $\SO(V)$, we have  
\[
\SO(n) \subseteq \Stab^\circ_{\SO(V)}(h_3).
\] 
By the branching rules for the restriction from $\SO(n+1)$ to $\SO(n)$, see, e.g., \cite[Section~8.3]{GW09}, one verifies that $\calH_3$ contains a unique, up to scaling, $\SO(n)$-invariant, that is $h_3 = x_0^3$. We conclude 
\[
f = \mu x_0^3 + \lambda x_0 q = x_0 (\mu x_0^2 + \lambda q)
\]
 as desired.
\end{proof}

We conclude with a remark on the isotropy group of plane cubics.
\begin{remark}\label{rmk: plane cubics}
    The elements of $S^3 \bbC^3$ are completely classified up to the action of $\GL_3$ as follows. There is a continuous $1$-dimensional family of smooth cubics, uniquely determined by a single invariant, and seven additional orbits. All smooth cubics, and the one whose singularity is a simple node have $0$-dimensional isotropy group. The cuspidal cubic has a $1$-dimensional isotropy group, whose intersection with $\SL_3$ is $0$-dimensional. The generic reducible cubic decomposes as the union of a conic and a generic line: then, as in \Cref{prop: stabilizer cubics}, there is a copy of $\SO(2)$ contained in the isotropy group and in $\SL_3$. All other cubics are reducible and contained in the orbit-closure of the generic reducible cubic, hence the intersection of their isotropy group with $\SL_3$ is positive-dimensional. This recovers Ruppert's result, showing that a plane cubic is reducible if and only if the intersection between its isotropy group and $\SL_3$ has positive dimension.
\end{remark}

\section{Determinantal equations for higher strength}\label{sec: more equations}

In this section, we generalize Ruppert's equations from \Cref{thm: ruppert original} providing equations for the variety $\sigma_r(\calR^\bullet_{n,d})$ for every $n,d$ when $n \geq 2r$. We point out that in this range \Cref{prop: singular locus} already gives equations for $\sigma_r(\calR^\bullet_{n,d})$ because forms of strength $r$ in $2r+1$ variables are singular. However, we expect the equations introduced in this section to be independent from the ones arising by the non-emptiness of the singular locus: in a way, they are a refinement of that condition, similarly to the situation described in \Cref{rmk: syz}. We will prove this for the case $(r,d) = (2,3)$ in \Cref{thm: 11 syzygies for slrk 2}.

Fix $r$, let $n=2r$ and $V$ be a vector space with $\dim V = n+1$. Given homogeneous polynomials $g_1 \vvirg g_r,h_1 \vvirg h_r$ with $\deg(g_i) + \deg(h_i) = d$, consider the following matrix of size $(2r+1) \times (2r+1)$
\[
M\coloneqq\begin{pNiceMatrix}[columns-width = auto,cell-space-limits=3pt,xdots/shorten = 0.5 em]
\partial_0 g_1 &  \Cdots  & \partial_{2r} g_1 \\
\Vdots &\Ddots & \Vdots \\
\partial_0 g_r &  \Cdots  & \partial_{2r} g_r \\
\partial_0 h_1 &  \Cdots & \partial_{2r} h_1 \\
\Vdots & \Ddots & \Vdots \\
\partial_0 h_r &  \Cdots & \partial_{2r} h_r \\
\partial_0 &  \Cdots& \partial_{2r} \\
\end{pNiceMatrix}
\]
and the tensor
\[
\Delta \coloneqq \det M \in S^{r(d-2)} V \otimes V^*,
\]
arising, as in the proof of \Cref{prop: ruppert_first_implication_new}, from the Laplace expansion of the determinant with tensor products between the last row and the others.

\begin{lemma}\label{lemma: Delta is in prolong sl}
 For every $g_1 \vvirg g_r$ and every $h_1 \vvirg h_r$ with $\deg(g_i) + \deg(h_i) = d$, we have 
 \[
 \Delta \in \fraksl^{(r(d-2)-1)}(V) \subseteq  S^{r(d-2)}V \otimes V^*.
 \] 
\end{lemma}
\begin{proof}
 We show that the contraction of the $V^*$ factor on $S^{r(d-2)}V$ maps $\Delta$ to $0$. For $i = 1 \vvirg r$, denote $g_{r+i} :=h_i$. Then, up to a sign,
 \[
  \Delta = \sum_{\sigma \in \frakS_{2r+1}} (-1)^\sigma  \biggl( \prod_{i=1}^{2r} \partial_{\sigma(i)} g_i \biggr) \otimes \partial_{\sigma(0)}.
 \]
 Therefore, the result of the contraction map applied to $\Delta$ is
 \begin{align*}
  \sum_{\sigma \in \frakS_{2r+1}} (-1)^\sigma  \partial_{\sigma(0)} \biggl(\prod_{i=1}^{2r} \partial_{\sigma(i)} g_i \biggr) &=
  \sum_{\sigma\in \frakS_{2r+1}} (-1)^\sigma \sum_{i=1}^{2r} \bigl(\partial_{\sigma(0)} \partial_{\sigma(i)} g_i\bigr) \cdot  \prod_{j \neq i} \partial_{\sigma(j)} g_j \\
  &=\sum_{i=1}^{2r} \sum_{\sigma \in \frakS_{2r+1}} (-1)^{\sigma} \bigl(\partial_{\sigma(0)} \partial_{\sigma(i)} g_i\bigr) \cdot  \prod_{j \neq i} \partial_{\sigma(j)} g_j.
 \end{align*}
For every $i$ and every $\sigma\in\frakS_{2r+1}$, the term in the inner summation corresponding to $\sigma$ is opposite to the one corresponding to $\sigma \circ (0,i)$ where $(0,i)$ is the permutation swapping $0$ and $i$. Therefore the summation is zero.
\end{proof}

\begin{lemma}\label{lemma: Delta kills f}
 For every $g_1 \vvirg g_r$ and every $h_1 \vvirg h_r$ with $\deg(g_i) + \deg(h_i) = d$, let $f = g_1 h_1 + \cdots + g_r h_r$. Then $\Delta (f) =0$. 
\end{lemma}
\begin{proof}
Let $M(f)$ be the result of contracting the differentials in the last row of $M$ against $f$. Then $\Delta(f) = \det\bigl( M(f)\bigr)$ and the $j$-th element of the last row of $M(f)$ is 
\[
 \partial_j(f) = \sum_{i=1}^r h_i \partial_j(g_i)  + \sum_{i=1}^r g_i \partial_j(h_i) .
\]
By linearity of the determinant in the last row, we have
\[
\det M(f) = \sum_{i=1}^r h_i \det M(g_i) + \sum_{i=1}^r g_i \det M(h_i) = 0
\]
because all the matrices $M(g_i)$ and $M(h_i)$ have two equal rows.
\end{proof}
\Cref{lemma: Delta is in prolong sl} and \Cref{lemma: Delta kills f} show that $\Delta$ is an element of the kernel of the Ruppert map
\[
 \rho^{(r(d-2))}_f : \fraksl^{(r(d-2)-1)}(V) \to S^{r(d-2)+(d-1)} V.
\]
Moreover, it is clear that if $f$ is generic in any of the components of $\sigma_r(\calR^\bullet_{n,d})$, the corresponding $\Delta$ is nonzero. This can give, in principle, determinantal equations for the secant variety $\sigma_{r}(\calR^\bullet_{n,d})$, similarly to \Cref{prop: ruppert_first_implication_new}: for an element in $f \in \sigma_{r}(\calR^\bullet_{n,d})$, the Ruppert map $\rho^{(r(d-2))}_f$ is expected to have rank strictly smaller than the one attained by a generic $f \in S^d V$. There are however two difficulties in completing a proof. As explained in \Cref{rmk: syz}, the kernel of $\rho^{(r(d-2))}_f$ is the space of syzygies of $\langle \partial_0  f \vvirg \partial_n f\rangle$ of degree $r(d-2)$ which are contained in $\fraksl^{(r(d-2)-1)}(V)$. When $f$ is generic, these syzygies are generated by the Koszul relations; we expect that if $f \in \sigma_r(\calR^\bullet_{n,d})$ then $\Delta$ defines a syzygy which is linearly independent from the subspace generated by the Koszul syzygies of $f$ but in principle it is possible that one of the Koszul syzygies degenerates to $\Delta$, leaving the rank unchanged. Moreover, the rank of $\rho^{(r(d-2))}_f$ for generic $f$ is not straightforward to determine: this amounts to computing the dimension of the intersection  between the space generated by the Koszul syzygies and the prolongation $\fraksl^{(r(d-2)-1)}(V)$. The construction can be, however, carried out explicitly when $(r,d) = (2,3)$, allowing us to prove that indeed the additional syzygy induced by $\Delta$ is independent from the Koszul syzygies. This case is addressed in the next section.

\subsection{Cubics of slice rank two}\label{sec: cubics slice rank 2}

We study the variety of cubics of slice rank $2$ in $\bbP S^3 \bbC^5$; indeed $n+1 = 2r+1=5$ is the smallest number of variables for which the construction described above yields equations. In this case, $r(d-2) = 2 \cdot (3-2) = 2$ and the syzygy $\Delta$ has the same degree as the generators of the Koszul module. This allows us to easily prove that it is independent from the Koszul syzygies. More precisely, we have the following result.
\begin{theorem}\label{thm: 11 syzygies for slrk 2}
  Let $\dim V = 5$ and let $f \in S^3 V$. Consider the Ruppert map
 \[
 \begin{tikzcd}[row sep=0pt,column sep=1pc]
  \rho^{(2)}_f \colon  \fraksl^{(1)}(V) \arrow{r} & S^{4} V.
 \end{tikzcd}
 \]
 If $f$ is generic, then $\rank( \rho^{(2)}_f) = 60$, whereas if $\slrk(f) \leq 2$ then $\rank(  \rho^{(2)}_f) \leq 59$. In particular, the minors of size $60$ of $ \rho^{(2)}_f$ define equations of degree $60$ for $\sigma_2(\calR_{4,3})$.
 \end{theorem}
\begin{proof}
The proof relies on the computation of the quadratic syzygies of the Jacobian ideal of $f$. 

 Let $f \in S^3 V$ be generic. Then $Z(f)$ is a smooth hypersurface and $\partial_0 f \vvirg \partial_4f$ form a regular sequence. The only quadratic syzygies are the Koszul syzygies, which, regarded as elements of $S^2 V \otimes V^*$ have the form 
    \[
    k_{ij}(f) = (\partial_i f) \otimes \partial_j - (\partial_j f) \otimes \partial_i
    \]
    with $i < j$. Since $\partial_{ij}f = \partial_{ji}f$, we have $k_{ij}(f) \in \fraksl^{(1)}(V)$ and they span $\ker \rho^{(2)}_f$. By \eqref{eqn: dim slpV}, $\dim \fraksl^{(1)}(V) = 70$,  therefore $\dim \langle k_{ij}(f) : i,j = 0 \vvirg 4 \rangle = \binom{5}{2} = 10$ implies $\rank \rho^{(2)}_f = 60$ for generic $f$.

Define 
\[
\begin{array}{lll}
\ell_{0,\eps} = x_0 + \eps \ell_0' &\qquad &q_{0,\eps} = (x_2^2+x_3^2+x_4^2) + \eps q_2'\\ 
\ell_{1,\eps} = x_1 + \eps \ell_1' & \qquad & q_{1,\eps} = (x_2x_3 + x_1x_4)+ \eps q_1'
\end{array}
\]
and let 
\[
f_\eps = \ell_{0,\eps} q_{0,\eps} + \ell_{1,\eps} q_{1,\eps}.
\]
For a generic choice of $\ell_0', \ell_1', q_0',q_1'$ and $\eps$, the cubic $f_\eps$ is a generic element of $\sigma_2(\calR^1_{4,3})$. Let 
\[
\Delta_\eps = \det \left( \begin{array}{ccc}
    \partial_0 \ell_{0,\eps} & \cdots & \partial_4 \ell_{0,\eps} \\
    \partial_0 \ell_{1,\eps} & \cdots & \partial_4 \ell_{1,\eps} \\
    \partial_0 q_{0,\eps} & \cdots & \partial_4 q_{0,\eps} \\
    \partial_0 q_{1,\eps} & \cdots & \partial_4 q_{1,\eps} \\
    \partial_0  & \cdots & \partial_4  
\end{array}\right) \in  S^2 V \otimes V^*.
\]
By \Cref{lemma: Delta is in prolong sl} and \Cref{lemma: Delta kills f}, we have $\Delta_\eps \in \ker \rho^{(2)}_{f_\eps}$ for every $\eps$. Therefore $K_\eps = \langle k_{ij}(f_\eps) : i,j = 1 \vvirg 5 \rangle + \langle \Delta_\eps \rangle$ is a subspace of $\ker \rho^{(2)}_{f_\eps}$. 

We observe that $\dim K_\eps = 11$ for generic $\eps$. Clearly $\dim K_\eps \leq 11$ because it is spanned by $11$ elements of $S^2 V \otimes V^*$. Moreover, equality holds by semicontinuity, because for $\eps = 0$, one can directly verify with the support of a computer algebra software that $\dim K_0 = 11$.

This shows that if $\slrk(f) \leq 2$, then $\rho^{(2)}_{f}$ has a kernel of dimension at least $11$ and therefore $\rank \rho^{(2)}_{f} \leq 70-11 = 59$. This concludes the proof.
\end{proof}
It is natural to ask whether the equations described in \Cref{thm: 11 syzygies for slrk 2} define $\sigma_2(\calR^{1}_{4,3})$ set-theoretically. This is not the case. We describe a $21$-dimensional variety $\calS \subseteq \bbP S^3 V$ not contained in $\sigma_2(\calR^{1}_{4,3})$ with the property that, if $f \in \calS$, then $ \rank \rho^{(2)}_{f} \leq 53 \leq 59$.

Let $C_4 \subseteq \bbP V^*$ be a rational normal curve of degree $4$. Then $\sigma_2(C_4)$ is a cubic hypersurface. In coordinates, if $C_4 = \{ (t_0^4, t_0^3t_1 \vvirg t_1^4) \in \bbP V^* : (t_0,t_1) \in \bbP^1\}$, we have 
\[
\sigma_2(C_4) = Z(f) \qquad \text{ where } \qquad f = \det \left(\begin{array}{ccc} 
x_0 & x_1 & x_2 \\
x_1 & x_2 & x_3 \\
x_2 & x_3 & x_4 
\end{array}\right).
\]
A direct computation with the support of a computer algebra software shows that in this case $\rank \rho^{(2)}_{f} = 53$. On the other hand, $\slrk(f) \geq 3$: indeed, if $\slrk(f) \leq 2$, then by \Cref{prop: slice equivalent to fano}, $\sigma_2(C_4)$ would contain a linear space of codimension $2$. The next result shows that this is not possible.
\begin{lemma}\label{lemma: no P2 in secants}
Let $n\geq 4$ and let $C\subset \bbP^n$ be a smooth, irreducible non-degenerate curve. Then $\sigma_2(C)$ does not contain a linear space $P$ with $\dim P = 2$. 
\end{lemma}
\begin{proof}
Assume by contradiction that $\sigma_2(C)$ contains a linear space $P$ with $\dim P = 2$. Let $q \in P$ be a generic point and observe that $q$ lies on a secant line to $C$. Indeed, since $C$ is smooth, $\sigma_2(C)$ is union of the set of points lying on secant lines and the set of those lying on tangent lines. The latter is the tangential variety $\tau(C)$ of $C$, which is irreducible of dimension $2$. If a generic point $q \in P$ lied on $\tau(C)$, then $P = \tau(C)$ because they are both irreducible of dimension $2$. This shows $\tau(C)$ is a linear space, in contradiction with the linear non-degeneracy of $C$. Therefore $p$ lies on a secant line to $C$.

Define 
\[
E =\bar{ \{ (p_1,p_2) \in C \times C : P \cap  \langle p_1,p_2 \rangle \neq  \emptyset\}} \subseteq C \times C.
\]
Let $u : E \dashto P$ be the map defined by $u(p_1,p_2) = P \cap \langle p_1,p_2 \rangle$. By construction $u$ is dominant and since $P$ and $C \times C$ both have dimension $2$ we obtain that a generic secant line intersects $P$. Therefore, for a generic $q \in C$ and a generic secant line $\langle q, p\rangle$ with $p \in C$, we have $\langle q, p\rangle \cap P \neq \emptyset$, so $p \in \langle P , q\rangle$. This shows $C \subseteq \langle P, q\rangle$ in contradiction with the non-degeneracy of $C$.
\end{proof}

\Cref{lemma: no P2 in secants} guarantees that if $Z(f) = \sigma_2(C_4)$ then $\slrk(f) \geq 3$. In particular, for every rational normal quartic curve $C_4 \subseteq \bbP V^*$, the cubic defining $\sigma_4(C_4)$ satisfies $\rank \rho_f^{(2)} \leq 59$ but $\slrk(f) \geq 3$. Let $\calS \subseteq \bbP S^3 V$ be the (closure of) the set of such cubic polynomials. Since the rational normal quartic is unique up to the action of $\SL_5$, we have 
\[
\dim \calS = \dim \SL_5 - \dim \SL_2 = 24 - 3 = 21,
\]
where $\SL_2$ is the subgroup preserving a fixed rational normal curve.

\section{A reduction for cubics of slice rank two}
In this final section, we prove an inheritance result similar to \autoref{thm: ruppert any vars} for cubic forms of slice rank $2$.
\begin{theorem}\label{thm:restrictions_vars}
  Let $f \in S^3 V$ with $\dim V = n+1$ irreducible. The following statements are equivalent:
  \begin{enumerate}[label=(\roman*)]
      \item $\slrk(f) = 2$;
      \item for a generic $E \subseteq V^*$ with $\dim E = 14$, $\slrk(f|_E) = 2$;
      \item for every $E \subseteq V^*$ with $\dim E = 14$, $\slrk(f|_E) \leq 2$.
  \end{enumerate}
\end{theorem}
We do not expect the bound of fourteen variables in \Cref{thm:restrictions_vars} to be optimal. We will propose the more general \Cref{conj: general restrictions}, which predicts that in the case of slice rank $2$ the optimal reduction would be to five variables. However, as we mentioned in the introduction, providing explicit upper bounds for such reduction results is in general very challenging. 

We record an immediate fact regarding singularities of secant varieties.
\begin{lemma}\label{lem: secant of sing}
    Let $f \in S^3 V$. If $Y \subseteq \Sing\bigl(Z(f)\bigr)$, then $\sigma_2(Y) \subseteq Z(f)$.
\end{lemma}
\begin{proof}
Let $L$ be a secant line to $Y$ so that $L \cap Y$ consists of (at least) two points. Since $Y$ is in the singular locus of $Z(f)$, the intersection multiplicity between $L$ and $Z(f)$ at each point of $Y\cap L$ is at least $2$. Since $\deg(f) = 3$, we obtain $L \subseteq Z(f)$ by B\'{e}zout's Theorem. Taking closures, we have
\[
\sigma_2(Y)\subseteq Z(f)
\] 
as desired.
\end{proof}
We also recall the classical Palatini's Lemma and a stronger version useful in the proof of \Cref{thm:restrictions_vars}. We refer to \cite[Proposition 1.2.2(3)]{Russo:GeometrySpecialProjVars} for the classical statement and to \cite[Corollary~3.4.2]{Russo:GeometrySpecialProjVars} for its stronger version. 
\begin{proposition}\label{prop: palatini upgraded}
Let $X \subseteq \bbP^N$ be a linearly non-degenerate irreducible variety. Then either $X$ is a hypersurface and $\sigma_2(X) = \bbP^N$ or 
\begin{equation}\label{eq: palatini statement}
\dim \bigl(\sigma_2(X)\bigr) \geq \dim X + 2.
\end{equation}
Moreover, if equality holds in \eqref{eq: palatini statement}, then one of the following holds:
\begin{enumerate}[label=(\roman*)]
 \item $\dim X = N-2$ and $\sigma_2(X) = \bbP^N$;
 \item $X$ is a curve or a cone over a curve;
 \item $X$ is the Veronese surface $\nu_2(\bbP^2)$ or a cone over it.
\end{enumerate}
\end{proposition}

We will use these two results, together with \Cref{thm: 11 syzygies for slrk 2}, to complete the proof of \Cref{thm:restrictions_vars}.

\begin{proof}[Proof of \Cref{thm:restrictions_vars}]
The implications $(i)\Rightarrow(iii)$ and $(iii)\Rightarrow(ii)$ are clear.

We prove the implication $(ii)\Rightarrow(i)$. In other words, we prove that under the assumption of $(ii)$, $Z(f)$ contains a linear space of codimension at most $2$. Without loss of generality, assume $\dim V \geq 15$ otherwise the statement is immediately verified. Moreover, assume $f$ is concise in $S^3 V$, in the sense that there is no change of coordinates such that $f$ can be written in fewer than $(\dim V)$-many variables.

  Let $E \subseteq V^*$ be a generic linear space of dimension $14$; then $\slrk(f|_E) \leq 2$. By \Cref{prop: singular locus}, we have that $Z(f|_E) \subseteq \bbP E$ is singular in codimension (at most) $4$. On the other hand, by Bertini's Theorem and the genericity of $E$, we have \[\Sing( f|_E) = \Sing(f) \cap \bbP E.\] Therefore $Z(f)$ is singular in codimension at most $4$ as well.
 
  The rest of the proof is a case-by-case analysis on the dimension and the geometry of $\Sing(f)$. Let $Y$ be an irreducible component of $\Sing(f)$ of maximal dimension. By \Cref{lem: secant of sing}, we have $\sigma_2(Y) \subseteq Z(f)$. We recall that if $f$ is singular along $Y$, then $f \in I(Y)^{(2)}$, where $-^{(2)}$ denotes the second symbolic power, see \cite[Theorem~3.15]{Eis:CA}. In all cases where this fact will be used, $I(Y)$ is generated by a regular sequence and, therefore, $I(Y)^{(2)} = I(Y)^2$, e.g., \cite[(2.1)]{hochster1973criteria}. We have the following cases:
\begin{enumerate}[label=(\arabic*), leftmargin=20pt]
  \item \underline{$\codim(Y) = 1$}. In this case, $f$ is not reduced, therefore $\slrk(f) = 1$.
\item  \underline{$\codim(Y) = 2$}. If $Y$ is a linear space, then $\slrk(f) \leq 2$ because $Y  \subseteq Z(f)$. 
  If $Y$ is not a linear space, since $\sigma_2(Y) \subseteq Z(f)$, we have
  \[
  \dim \sigma_2(Y) = \dim Y + 1.
  \] 
  By \Cref{prop: palatini upgraded}, we deduce that $\sigma_2(Y) = \langle Y \rangle$ is a linear space of codimension $1$. In this case $f$ is reducible and $\slrk(f) = 1$. 
\item \underline{$\codim (Y) = 3$}. We consider three subcases.
\begin{enumerate}[leftmargin=\parindent]
  \item If $\codim \sigma_2(Y) = 3$, then we have $\sigma_2(Y)=Y$, that is, $Y$ is a linear space. Assume 
  \[
  Y = Z(x_0,x_1,x_2).
  \]
  Since $f$ is singular along $Y$, we deduce \[f \in (x_0,x_1,x_2)^2.\] Therefore $f$ is a combination of the six quadratic monomials in $x_0,x_1,x_2$. We conclude that $f$ can be written in at most $9$ variables after a change of coordinates, in contradiction with the conciseness assumption.
\item If $\codim \sigma_2(Y) = 2$, then $\sigma_2(Y) \subseteq Z(f)$ is a linear space by \Cref{prop: palatini upgraded}. In this case, we conclude immediately $\slrk(f) \leq 2 $.
\item  If $\codim \sigma_2(Y) = 1$, then since $\sigma_2(Y)\subseteq Z(f)$, $\sigma_2(Y)$ is a component of $Z(f)$.
If $\sigma_2(Y)$ is a linear space, then $\slrk(f) = 1$. If $\sigma_2(Y)$ is a quadratic hypersurface, then $f$ would be reducible as well, and hence $\slrk(f) = 1$; in fact, this case yields a contradiction because $\sigma_2(Y)$ cannot be a quadratic hypersurface; see, e.g., \cite[Theorem 1.2]{SidSul:ProlontationCompAlgebra}. Therefore, the only other possibility is that $\sigma_2(Y) = Z(f)$ and we have 
\[
\dim \sigma_2(Y) = \dim Y  + 2.
\]
By \Cref{prop: palatini upgraded}, we deduce that $Y$ is either a cone over a curve in $\bbP^4$ or a cone over $\nu_2(\bbP^2) \subseteq \bbP^5$. In this case, also $Z(f)$ is a cone over a hypersurface in $\bbP^4$ or $\bbP^5$ because secant varieties of cones are themselves cones. In both cases, this is in contradiction with the conciseness assumption.
\end{enumerate}
\item \underline{$\codim (Y) = 4$}. In this case, we first show that $Y$ can be chosen to be contained in a linear space of codimension $2$. This is a condition similar to the one of \Cref{prop: singular locus compt}. Let $Y_1 \vvirg Y_p$ be the irreducible components of $Z(f)$ having codimension exactly $4$. By Bertini's Theorem, for a generic linear space $H \subseteq V^*$ with $\dim H = 6$, we have that $Y_j ' = Y_j \cap H$ are the unique irreducible components of dimension $1$ of the singular locus of $Z(f|_H) = Z(f) \cap \bbP H$. By assumption $\slrk(f|_H) = 2$ so 
\[
f|_H = \ell_0 h_0 + \ell_1 h_1
\] 
and we have 
\[
H \cap Z(\ell_0,\ell_1,h_0,h_1) \subseteq \Sing(f|_H);
\]
since 
\[
\dim  \bigl(H \cap Z(\ell_0,\ell_1,h_0,h_1)\bigr) \geq 1,
\]
equality holds and we deduce that the irreducible components of $ H \cap Z(\ell_0,\ell_1,h_0,h_1)$ are some of the components $Y_1' \vvirg Y_p'$. Let $Y = Y_j$ for an index $j$ such that $Y_j'$ is one of this components. We obtain that $Y$ has the property that, for a generic choice of $H$, $Y \cap H$ is contained in a linear space of codimension $2$. By genericity, we deduce that $Y$ is contained in a linear space of codimension $2$ as well: this is a consequence, for instance, of \cite[Proposition A3]{GesKayTel:ChoppedIdeals}, which guarantees that if $Y$ is linearly non-degenerate then generic $0$-dimensional linear sections of $Y$ are linearly non-degenerate as well. Therefore, there are only three possibilities for $\codim \sigma_2(Y)$:
\begin{enumerate}[leftmargin=\parindent]
\item If $\codim \sigma_2(Y) = 4$, then $Y$ is a linear space. Assume $Y = Z(x_0,x_1,x_2,x_3)$. Since $f$ is singular along $Y$, we deduce $f \in (x_0,x_1,x_2,x_3)^2$. Therefore $f$ is a combination of the ten quadratic monomials in $x_0,x_1,x_2,x_3$. We conclude that $f$ can be written in at most $10+4$ variables after a change of coordinates, in contradiction with the conciseness assumption.

\item If $\codim \sigma_2(Y) = 3$, then $\sigma_2(Y) = \langle Y \rangle$ by \Cref{prop: palatini upgraded} and $Y$ is contained in a linear space of codimension $3$. Therefore $I(Y) = (\ell_0,\ell_1,\ell_2,g)$ for some linear forms $\ell_0 \vvirg \ell_2$ and an irreducible form $g$ of degree at least $2$. Since $f$ is singular along $Y$, we have $f \in I(Y)^2$. If $\deg (g) \geq 3$, we deduce $f \in (\ell_0,\ell_1,\ell_2)^2$ and we conclude that $f$ can be written in at most $10$ variables, in contradiction with the conciseness assumption. If $\deg(g) = 2$, let $\ell_0=x_0, \ell_1=x_1,\ell_2=x_5$ and choosing three additional linear forms $x_2,x_3,x_4$, we may assume
\[
f = x_0^2 x_2+x_0x_1 x_3 + x_1^2 x_4 + x_5 g .
\]
where $g$ is a quadric in the variables $x_0,\ldots, x_n$.

We will prove that $g|_{x_0 = \cdots = x_5 = 0}$ is a quadric of rank at least $2$. This will follow from a dimension count based on the fact that $f$ is concise in at least $15$ variables. Write 
\[
V^* = W_1 \oplus \langle \partial_5 \rangle \oplus W_2 
\]
with $W_1 = \langle \partial_0 \vvirg \partial_4\rangle$ and $W_2 = \langle \partial_6 \vvirg \partial_n\rangle$. The conciseness assumption guarantees that $\dim (S^2 V^* \contract f ) \geq 15$; here $\contract$ denotes the differentiation action. We are going to show 
\[
\dim [(W_2 \contract g ) / \langle x_0 \vvirg x_5 \rangle]  \geq 2,
\]
where for any spaces $U_1,U_2 \subseteq U$, we write $U_1/U_2$ to denote $(U_1+U_2)/U_2$. Write 
\[
S^2 V^* = V^*\cdot \partial_5 + S^2 W,
\]
where $W = W_1 \oplus W_2$. Now, the variables $x_6,\dots,x_{n}$ only appears in $g$ and with degree at most $2$, so 
\[
 S^2 W \contract f \subseteq \langle x_0 \vvirg x_5 \rangle,
 \]
 which implies $\dim (S^2 W \contract f ) \leq 6$. Therefore, we have 
\begin{equation}\label{eqn: at least 9 dimensions}
\dim \bigl( (V^* \contract \partial_5 f) / \langle x_0 \vvirg x_5 \rangle \bigr) \geq 9.
\end{equation}
Since
\[
\partial_5 f = x_5 \partial_5g + g \quad \text{and} \quad  W \contract (x_5 \partial_5g) \subseteq \langle x_5 \rangle,
\]
the latter does not contribute modulo $\langle x_0 \vvirg x_5\rangle$.
We deduce that the contribution in \eqref{eqn: at least 9 dimensions} is given by (a subspace of) 
\[
\langle \partial_5^2 ((x_5g) \rangle + (W \contract g).
\]
Since 
\[
\dim \langle \partial_5^2 (x_5g) \rangle \leq 1 \quad \text{ and } \quad \dim (W_1 \contract g) \leq 5
\]
we conclude that 
\[
\dim (W_2 \contract g)/\langle x_0 \vvirg x_5 \rangle \geq 3
\] 
which is at least $2$ as desired. This guarantees that $g|_{x_0 = \cdots = x_5 = 0}$ is a quadric of rank at least $2$. We may assume, without loss of generality, that 
\[
g|_{x_i = 0, i\neq 6,7} = x_6 x_7.
\] 
Let $H = \langle x_8 \vvirg x_n\rangle^\perp \subseteq V^*$. Then 
\[
f|_H = x_0^2 x_2+x_0x_1 x_3 + x_1^2 x_4 + x_5 g|_H
\]
satisfies $\slrk(f|_H) \leq 2$ because $\dim H = 8 \leq 14$. Indeed, by semicontinuity of slice rank, the assumption that restrictions of $f$ to generic linear subspaces of dimension at most $14$ have slice rank at most $2$ implies that the same holds for every subspace of dimension at most $14$. By construction, the monomial $x_6x_7$ appears in $g|_H$ whereas $x_6^2,x_7^2$ do not. We prove $\slrk(f|_H) \geq 3$, which yields a contradiction. Consider the group element $h_\eps \in \GL(V)$ defined by 
\begin{align*}
h_\eps (x_5) &=  \eps^2 x_5 \\
 h_\eps(x_6) &=  \eps^{-1} x_6 \\
  h_\eps(x_7) &=  \eps^{-1} x_7 \\
  h_\eps (x_i) &= x_i \text{ for } i \neq 5,6,7,
\end{align*}
and set
\[
f_0 = \lim_{\eps \to 0} h_\eps \cdot f|_H = x_0^2 x_2 + x_0 x_1 x_3 + x_1^2 x_4 + x_5 x_6 x_7 .
\]
Since $\slrk(f|_H) \leq 2$, we have $\slrk(f_0) \leq 2$ as well. However, $f_0$ does not satisfy the degree $60$ equations of \Cref{thm: 11 syzygies for slrk 2} for slice rank $2$. Explicitly, the restriction $f_0'$ of $f_0$ defined by 
\[
x_5 \mapsto x_0 + x_2, \qquad x_6 \mapsto x_3, \qquad x_7 \mapsto x_1 + x_4
\]
satisfies $\rank \rho^{(2)}_{f_0'} = 60$. This yields a contradiction.

\item If $\codim \sigma_2(Y) = 2$, then $\sigma_2(Y) = \langle Y \rangle$ is a linear space, and $\slrk(f) = 2$. 
\end{enumerate}
\end{enumerate}
This concludes the proof that condition (ii) implies condition (i).
\end{proof}

In the setting discussed in \Cref{subsec: functors}, \Cref{thm:restrictions_vars} provides the upper bound $n_0+1\leq 14$ for the dimension of a space for which set-theoretic equations for $\sigma_2(\calR^1_{n,3})$ can be obtained via pullback from $\sigma_2(\calR^1_{n_0,3})$. We do not expect the bound $n_0+1 \leq 14$ to be sharp. In fact, a more involved argument in case (4.a) of the proof of \Cref{thm:restrictions_vars} shows the bound $n_0 +1 \leq 13$, which does not affect any other part of the proof; however, we do not expect $n_0 + 1 \leq 13$ to be sharp either.

Since every cubic surface contains lines, $\sigma_2(\calR^1_{3,3}) = \bbP S^3 V$ when $\dim V = 4$. If $\dim V = 5$, the variety $\sigma_2(\calR^1_{4,3})$ has codimension $4$ in $\bbP S^3 V  = \bbP ^{34}$ and some defining equations were provided in \Cref{thm: 11 syzygies for slrk 2}. The discussion in \Cref{sec: cubics slice rank 2} shows that the equations of \Cref{thm: 11 syzygies for slrk 2} do not cut out $\sigma_2(\calR^1_{4,3})$ set-theoretically. There is, at least, one other module of equations on $\bbP S^3 V$ vanishing on $\sigma_2(\calR^{1}_{4,3})$ and not on the variety $\calS$ of hypersurfaces which are secants of a rational normal quartic. We expect these additional equation, together with those of \Cref{thm: 11 syzygies for slrk 2}, to be enough to cut out $\sigma_2(\calR^{1}_{n,3})$ in any number of variables; in particular a complete system of set-theoretic equations would arise from restrictions to five variables, similarly to the case of three variables in \Cref{thm: ruppert any vars}. More generally, we propose the following.
\begin{conjecture}\label{conj: general restrictions}
    A system of set-theoretic equations for $\sigma_r(\calR^1_{n,3})$ arises from restrictions to $2r+1$ variables. More precisely for $f \in S^3 V$, the following statements are equivalent:
      \begin{enumerate}[label=(\roman*)]
      \item $\slrk(f) \leq r$;
      \item for a generic $E \subseteq V^*$ with $\dim E \leq 2r+1$, $\slrk(f|_E) \leq r$;
      \item for every $E \subseteq V^*$ with $\dim E \leq 2r+1$, $\slrk(f|_E) \leq r$.
  \end{enumerate}
\end{conjecture}
\Cref{thm: ruppert any vars} proves the conjecture in the case $r=1$ (and any degree) and clearly the implications $(i)\Rightarrow(iii)$ and $(iii)\Rightarrow(ii)$ always hold. It is possible that an analog of \Cref{conj: general restrictions} holds for higher degree as well but the non-closedness property of strength for higher degree might give rise to peculiar counterexamples. 

Finally, it would be interesting to study other geometric properties of the components of $\sigma_r(\calR^\bullet_{n,d})$. Their degree is easy to compute in the case $r=1$. The degree of the component $\sigma_r(\calR^1_{n,d})$ can be expressed as the degree of a certain Chern class over a suitable Grassmannian \cite{Man:HypersurfacesEspacesLineaires} but we do not know an explicit closed formula. The degrees of the other irreducible components of $\sigma_r(\calR^\bullet_{n,d})$ seem hard to compute. For instance, in the case $(n,d,r) = (3,4,2)$, that is for quartic surfaces of strength two, the variety $\sigma_2(\calR^\bullet_{3,4})$ has three components
\[
\sigma_2(\calR^1_{3,4}), \qquad \bfJ(\calR^1_{3,4},\calR^2_{3,4}), \qquad \sigma_2(\calR^2_{3,4}),
\]
where $\bfJ(-,-)$ denotes the geometric join \cite[Section 6.17]{Harris}. The three components are hypersurfaces. The variety $\sigma_2(\calR^1_{3,4})$ of quartic surfaces containing a line is a hypersurface of degree $320$ \cite{Man:HypersurfacesEspacesLineaires}. Therefore, its defining equation is an $\SL(V)$-invariant in $S^{320} S^4 \bbC^4$ and computing it via a direct interpolation method is far out of reach for current technology. We have numerical evidence that $\bfJ(\calR^1_{3,4},\calR^2_{3,4})$ is a hypersurface of degree $2508$, whereas the computation of $\deg \sigma_2(\calR^2_{3,4})$ was beyond what we could achieve via standard monodromy methods \cite{BreTim}. We leave further study of these varieties for future work.

\subsection*{Acknowledgments}
The authors thank Daniele Agostini, Edoardo Ballico, Alessandro Danelon, Jan Draisma, J. M. Landsberg and Giorgio Ottaviani for their helpful comments and countless suggestions. Part of this paper was written while C.F. was a research fellow at the {University of Florence}. Part of this work was developed while C.F., F.G. and A.O. were visiting the Max Planck Institute for Mathematics in the Sciences, in Leipzig, during the conference MEGA 2024: the authors thank MPI and the MEGA organizers for their hospitality and for providing a stimulating research environment. 

C.F., A.O and E.V. are members of the research group GNSAGA of INdAM. 

C.F. has been partially supported by the scientific project \textit{Multilinear Algebraic Geometry} of the program \textit{Progetti di ricerca di Rilevante Interesse Nazionale} (PRIN), Grant Assignment Decree No.~973, adopted on 06/30/2023 by the Italian Ministry of University and Research (MUR), and by the project \textit{Thematic Research Programmes}, Action I.1.5 of the program \textit{Excellence Initiative -- Research University} (IDUB) of the Polish Ministry of Science and Higher Education. A.O. and E.V. acknowledge that the preparation of this article was partially funded by the Italian Ministry of University and Research in the framework of the Call for Proposals for scrolling of final rankings of the PRIN 2022 call through the project {\it Applied Algebraic Geometry of Tensors}, Protocol no. 2022NBN7TL.

\newpage

{
\bibliographystyle{alphaurl}
\bibliography{slicerank}

\newcommand{\etalchar}[1]{$^{#1}$}
\begin{thebibliography}{vdBDG{\etalchar{+}}25}

\bibitem[AC07]{AbCh07}
A.~Abdesselam and J.~V. Chipalkatti.
\newblock {Brill–Gordan loci, transvectants and an analogue of the Foulkes
  conjecture}.
\newblock {\em Advances in Mathematics}, 208(2):491–520, 2007.
\newblock \href {https://doi.org/10.1016/j.aim.2006.03.003}
  {\path{doi:10.1016/j.aim.2006.03.003}}.

\bibitem[AH20]{AnHoc:StillmanConj}
T.~Ananyan and M.~Hochster.
\newblock {Small subalgebras of polynomial rings and Stillman’s conjecture}.
\newblock {\em J. Amer. Math. Soc.}, 33(1):291–309, 2020.
\newblock \href {https://doi.org/10.1090/jams/932}
  {\path{doi:10.1090/jams/932}}.

\bibitem[BBOV22]{BBOV:StrengthNotClosed}
E.~Ballico, A.~Bik, A.~Oneto, and E.~Ventura.
\newblock {The set of forms with bounded strength is not closed}.
\newblock {\em Comptes Rendus. Mathématique}, 360(G4):371–380, 2022.
\newblock \href {https://doi.org/10.5802/crmath.302}
  {\path{doi:10.5802/crmath.302}}.

\bibitem[BBOV23]{BBOV23}
E.~Ballico, A.~Bik, A.~Oneto, and E.~Ventura.
\newblock {Strength and slice rank of forms are generically equal}.
\newblock {\em Israel J. Math.}, 254(1):275--291, 2023.
\newblock \href {https://doi.org/10.1007/s11856-022-2397-0}
  {\path{doi:10.1007/s11856-022-2397-0}}.

\bibitem[BDDE22]{BDDE22}
A.~Bik, A.~Danelon, J.~Draisma, and R.~H. Eggermont.
\newblock {Universality of high-strength tensors}.
\newblock {\em Vietnam Journal of Mathematics}, 50(2):557–580, 2022.
\newblock \href {https://doi.org/10.1007/s10013-021-00522-7}
  {\path{doi:10.1007/s10013-021-00522-7}}.

\bibitem[BDE19]{BDK19}
A.~Bik, J.~Draisma, and R.~H. Eggermont.
\newblock {Polynomials and tensors of bounded strength}.
\newblock {\em Comm. Cont. Math.}, 21(07):1850062, 2019.
\newblock \href {https://doi.org/10.1142/S0219199718500621}
  {\path{doi:10.1142/S0219199718500621}}.

\bibitem[BDV24]{BDV24}
A.~Blatter, J.~Draisma, and E.~Ventura.
\newblock {Implicitisation and parameterisation in polynomial functors}.
\newblock {\em Found. Comput. Math.}, 24(5), 2024.
\newblock \href {https://doi.org/10.1007/s10208-023-09619-6}
  {\path{doi:10.1007/s10208-023-09619-6}}.

\bibitem[BG21]{BryGes}
T.~Brysiewicz and F.~Gesmundo.
\newblock {The Degree of Stiefel Manifolds}.
\newblock {\em Enum. Comb. Appl.}, 1(3):n.S2R20, 2021.
\newblock \href {https://doi.org/10.54550/ECA2021V1S3R20}
  {\path{doi:10.54550/ECA2021V1S3R20}}.

\bibitem[BO21]{BikOne:StrengthGeneralPoly}
A.~Bik and A.~Oneto.
\newblock {On the strength of general polynomials}.
\newblock {\em Lin. and Mult. Algebra}, page 1–27, 2021.
\newblock \href {https://doi.org/10.1080/03081087.2021.1947955}
  {\path{doi:10.1080/03081087.2021.1947955}}.

\bibitem[Bor91]{Borel:LinAlgGp}
A.~Borel.
\newblock {\em {Linear algebraic groups}}, volume 126 of {\em {Graduate Texts
  in Mathematics}}.
\newblock Springer, 1991.

\bibitem[BT18]{BreTim}
P.~Breiding and S.~Timme.
\newblock {HomotopyContinuation.jl: A package for homotopy continuation in
  Julia}.
\newblock In {\em {Mathematical Software–ICMS 2018: 6th International
  Conference, South Bend, IN, USA, July 24-27, 2018, Proceedings 6}}, page
  458–465. Springer, 2018.
\newblock \href {https://doi.org/10.1007/978-3-319-96418-8_54}
  {\path{doi:10.1007/978-3-319-96418-8_54}}.

\bibitem[CCG08]{CCG08}
E.~Carlini, L.~Chiantini, and A.~V. Geramita.
\newblock {Complete intersections on general hypersurfaces}.
\newblock {\em Michigan Math. J.}, 57:121--136, 2008.
\newblock \href {https://doi.org/10.1307/mmj/1220879400.pdf}
  {\path{doi:10.1307/mmj/1220879400.pdf}}.

\bibitem[CGG{\etalchar{+}}19]{catalisano2019secant}
M.~V. Catalisano, A.~V. Geramita, A.~Gimigliano, B.~Harbourne, J.~Migliore,
  U.~Nagel, and Y.-S. Shin.
\newblock Secant varieties of the varieties of reducible hypersurfaces in
  $\mathbb{P}^n$.
\newblock {\em J. Algebra}, 528:381--438, 2019.
\newblock \href {https://doi.org/10.1016/j.jalgebra.2019.03.014}
  {\path{doi:10.1016/j.jalgebra.2019.03.014}}.

\bibitem[CGZ23]{ChrGesZui:GapSubrank}
M.~Christandl, F.~Gesmundo, and J.~Zuiddam.
\newblock {A Gap in the Subrank of Tensors}.
\newblock {\em SIAM J. Appl. Alg. Geom.}, 7(4):742–767, 2023.
\newblock \href {https://doi.org/10.1137/22M1543276}
  {\path{doi:10.1137/22M1543276}}.

\bibitem[Chi04]{Ch04}
J.~V. Chipalkatti.
\newblock {Invariant equations defining coincident root loci}.
\newblock {\em Arch. Math. (Basel)}, 83(5):422–428, 2004.
\newblock \href {https://doi.org/10.1007/s00013-004-1191-z}
  {\path{doi:10.1007/s00013-004-1191-z}}.

\bibitem[DES17]{DerkEggSnow}
H.~Derksen, R.~Eggermont, and A.~Snowden.
\newblock {Topological noetherianity for cubic polynomials}.
\newblock {\em Algebra \& Number Theory}, 11(9):2197–2212, 2017.
\newblock \href {https://doi.org/10.2140/ant.2017.11.2197}
  {\path{doi:10.2140/ant.2017.11.2197}}.

\bibitem[DM98]{DebMan}
O.~Debarre and L.~Manivel.
\newblock {Sur la variété des espaces linéaires contenus dans une
  intersection complète}.
\newblock {\em Math. Ann}, 312:549–574, 1998.
\newblock \href {https://doi.org/10.1007/s002080050235}
  {\path{doi:10.1007/s002080050235}}.

\bibitem[Dra19]{Draisma:TopologicalNoetherianityPolyFunctors}
J.~Draisma.
\newblock {Topological Noetherianity of polynomial functors}.
\newblock {\em J. Amer. Math. Soc.}, 32(3):691–707, 2019.
\newblock \href {https://doi.org/10.1090/jams/923}
  {\path{doi:10.1090/jams/923}}.

\bibitem[Eis95]{Eis:CA}
D.~Eisenbud.
\newblock {\em {Commutative {A}lgebra: with a view toward algebraic geometry}},
  volume 150 of {\em {Graduate Texts in Mathematics}}.
\newblock Springer-Verlag, New York, 1995.

\bibitem[FH91]{FultonHarris}
W.~Fulton and J.~Harris.
\newblock {\em {Representation theory: a first course}}, volume 129 of {\em
  {Graduate Texts in Mathematics}}.
\newblock Springer-Verlag, New York, 1991.
\newblock \href {https://doi.org/10.1007/978-1-4612-0979-9}
  {\path{doi:10.1007/978-1-4612-0979-9}}.

\bibitem[Fla25]{Fla24}
C.~Flavi.
\newblock {Decompositions of powers of quadrics}.
\newblock {\em Dissertationes Mathematicae}, 602:pp. 108, 2025.
\newblock \href {https://doi.org/10.4064/dm231228-23-1}
  {\path{doi:10.4064/dm231228-23-1}}.

\bibitem[Fr{\"o}85]{Fro:Ineq}
R.~Fr{\"o}berg.
\newblock {An inequality for {H}ilbert series of graded algebras}.
\newblock {\em Math. Scand.}, 56(2):117–144, 1985.

\bibitem[GGIL22]{GGIL22}
F.~Gesmundo, P.~Ghosal, C.~Ikenmeyer, and V.~Lysikov.
\newblock {Degree-Restricted Strength Decompositions and Algebraic Branching
  Programs}.
\newblock In {\em {42nd IARCS Ann. Conf. Found. Soft. Tech. and TCS (FSTTCS
  2022)}}, volume 250 of {\em {Leibniz Int. Proc. Inf. (LIPIcs)}}, page
  20:1–20:15, Dagstuhl, Germany, 2022. Schloss Dagstuhl – Leibniz-Zentrum
  für Informatik.
\newblock \href {https://doi.org/10.4230/LIPIcs.FSTTCS.2022.20}
  {\path{doi:10.4230/LIPIcs.FSTTCS.2022.20}}.

\bibitem[GGU25]{GGU25}
F.~Gesmundo, A.~Grosdos, and A.~Uschmajew.
\newblock {Identifiability through special linear measurements}.
\newblock {\em arXiv:2505.24328}, 2025.

\bibitem[GHL25]{GHL25}
F.~Gesmundo, Y.-I. Han, and B.~Lovitz.
\newblock {Linear preservers of secant varieties and other varieties of
  tensors}.
\newblock {\em J. Symbolic Comp.}, 131:102449, 2025.
\newblock \href {https://doi.org/10.1016/j.jsc.2025.102449}
  {\path{doi:10.1016/j.jsc.2025.102449}}.

\bibitem[GKKS13]{GKKS}
A.~Gupta, P.~Kamath, N.~Kayal, and R.~Saptharishi.
\newblock {Arithmetic circuits: A chasm at depth three}.
\newblock {\em Electronic Colloquium on Computational Complexity (ECCC)},
  20:26, 2013.
\newblock \href {https://doi.org/10.1109/FOCS.2013.68}
  {\path{doi:10.1109/FOCS.2013.68}}.

\bibitem[GKT25]{GesKayTel:ChoppedIdeals}
F.~Gesmundo, L.~Kayser, and S.~Telen.
\newblock {Hilbert Functions of Chopped Ideals}.
\newblock {\em J. Algebra}, 666:415–445, 2025.
\newblock \href {https://doi.org/10.1016/j.jalgebra.2024.11.017}
  {\path{doi:10.1016/j.jalgebra.2024.11.017}}.

\bibitem[GS]{M2}
D.~R. Grayson and M.~E. Stillman.
\newblock {Macaulay2, a software system for research in algebraic geometry}.
\newblock Available at \texttt{http://www.math.uiuc.edu/Macaulay2/}.
\newblock version 1.21.

\bibitem[Gua18]{Gua18}
Y.~Guan.
\newblock {Brill's equations as a {G}{L}({V})-module}.
\newblock {\em Lin. Alg. Appl.}, 548:273–292, 2018.
\newblock \href {https://doi.org/10.1016/j.laa.2018.02.026}
  {\path{doi:10.1016/j.laa.2018.02.026}}.

\bibitem[GW09]{GW09}
Roe Goodman and Nolan~R. Wallach.
\newblock {\em {Symmetry, representations, and invariants}}, volume 255 of {\em
  {Graduate Texts in Mathematics}}.
\newblock Springer, Dordrecht, 2009.

\bibitem[Har92]{Harris}
J.~Harris.
\newblock {\em {Algebraic geometry. A first course}}, volume 133 of {\em
  {Graduate Texts in Mathematics}}.
\newblock Springer-Verlag, New York, 1992.

\bibitem[Hil86]{HilbPowers}
D.~Hilbert.
\newblock {Über die notwendigen und hinreichenden kovarianten Bedingungen für
  die Darstellbarkeit einer binären Form als vollständiger Potenz}.
\newblock {\em Mathematische Annalen}, 27:158–161, 1886.

\bibitem[Ho73]{SiuHo73}
S.~M. Ho.
\newblock On the isotropic group of a homogeneous polynomial.
\newblock {\em Trans. Amer. Math. Soc.}, 183:495--498, 1973.
\newblock \href {https://doi.org/10.2307/1996482} {\path{doi:10.2307/1996482}}.

\bibitem[Hoc73]{hochster1973criteria}
M.~Hochster.
\newblock Criteria for equality of ordinary and symbolic powers of primes.
\newblock {\em Mathematische Zeitschrift}, 133(1):53--65, 1973.
\newblock \href {https://doi.org/10.1007/BF01226242}
  {\path{doi:10.1007/BF01226242}}.

\bibitem[IL03]{IvLan:Cartan}
T.~A. Ivey and J.~M. Landsberg.
\newblock {\em {Cartan for beginners}}, volume~61 of {\em {Graduate Studies in
  Mathematics}}.
\newblock American Mathematical Society, Providence, RI, 2003.

\bibitem[Jou83]{Jou:Bertini}
J.-P. Jouanolou.
\newblock {\em {Théoremes de Bertini et applications}}, volume~42.
\newblock Birkhäuser, 1983.

\bibitem[KM02]{KogMaz}
I.~A. Kogan and M.~M. Maza.
\newblock {Computation of canonical forms for ternary cubics}.
\newblock In {\em {ISSAC '02: Proc. 2002 Int. Symp. Symbolic and Alg. Comp.}},
  page 151–160, New York, NY, USA, 2002. ACM.
\newblock \href {https://doi.org/10.1145/780506.780526}
  {\path{doi:10.1145/780506.780526}}.

\bibitem[KZ18]{KZ18}
D.~Kazhdan and T.~Ziegler.
\newblock {On ranks of polynomials}.
\newblock {\em Algebras and Representation Theory}, 21(5):1017--1021, 2018.
\newblock \href {https://doi.org/10.1007/s10468-018-9783-7}
  {\path{doi:10.1007/s10468-018-9783-7}}.

\bibitem[Lan12]{Lan12}
J.~M. Landsberg.
\newblock {\em {Tensors: {G}eometry and {A}pplications}}, volume 128 of {\em
  {Graduate Studies in Mathematics}}.
\newblock American Mathematical Society, Providence, RI, 2012.

\bibitem[Lan17]{Lan17}
J.~M. Landsberg.
\newblock {\em {Geometry and complexity theory}}, volume 169 of {\em {Cambridge
  Studies in Advanced Mathematics}}.
\newblock Cambridge University Press, Cambridge, 2017.

\bibitem[LM24]{LM24}
E.~Liwski and F.~Mohammadi.
\newblock {Paving matroids: defining equations and associated varieties}.
\newblock {\em arXiv:2403.13718}, 2024.

\bibitem[LO13]{LanOtt}
J.~M. Landsberg and G.~Ottaviani.
\newblock {Equations for secant varieties of {V}eronese and other varieties}.
\newblock {\em Ann. Mat. Pura Appl.}, 192(4):569–606, 2013.
\newblock \href {https://doi.org/10.1007/s10231-011-0238-6}
  {\path{doi:10.1007/s10231-011-0238-6}}.

\bibitem[LP01]{LP01}
C.-K. Li and S.~Pierce.
\newblock {Linear {P}reservers {P}roblems}.
\newblock {\em Amer. Math. Monthly}, 108(7):591–605, 2001.
\newblock \href {https://doi.org/10.2307/2695268} {\path{doi:10.2307/2695268}}.

\bibitem[Man99]{Man:HypersurfacesEspacesLineaires}
L.~Manivel.
\newblock {Sur les hypersurfaces contenant des espaces linéaires}.
\newblock {\em Comptes Rendus de l'Académie des Sciences - Series I -
  Mathematics}, 328(4):307–312, 1999.
\newblock \href {https://doi.org/10.1016/S0764-4442(99)80215-8}
  {\path{doi:10.1016/S0764-4442(99)80215-8}}.

\bibitem[Nas20]{Nas20}
E.~Naslund.
\newblock {The partition rank of a tensor and $k$-right corners in
  $\mathbb{F}_q^n$}.
\newblock {\em Journal of Combinatorial Theory, Series A}, 174:105190, 2020.
\newblock \href {https://doi.org/10.1016/j.jcta.2019.105190}
  {\path{doi:10.1016/j.jcta.2019.105190}}.

\bibitem[Rup86]{Rup:ReduziabilitatKurven}
W.~Ruppert.
\newblock {Reduzibilität ebener {K}urven}.
\newblock {\em J. Reine Angew. Math.}, 369:167–191, 1986.
\newblock \href {https://doi.org/10.1515/crll.1986.369.167}
  {\path{doi:10.1515/crll.1986.369.167}}.

\bibitem[Rus16]{Russo:GeometrySpecialProjVars}
F.~Russo.
\newblock {\em {On the geometry of some special projective varieties}},
  volume~18 of {\em {Lecture Notes of the Un. Mat. Ital.}}
\newblock Springer, 2016.

\bibitem[Sch00]{Schin:PolySpecialRegardRed}
A.~Schinzel.
\newblock {\em {Polynomials with special regard to reducibility}}, volume~77.
\newblock Cambridge University Press, 2000.

\bibitem[SS09]{SidSul:ProlontationCompAlgebra}
J.~Sidman and S.~Sullivant.
\newblock {Prolongations and computational algebra}.
\newblock {\em Canad. J. Math.}, 61(4):930–949, 2009.
\newblock \href {https://doi.org/10.4153/CJM-2009-047-5}
  {\path{doi:10.4153/CJM-2009-047-5}}.

\bibitem[vdBDG{\etalchar{+}}25]{vdBDGIL:Metapolynomials}
M.~van~den Berg, P.~Dutta, F.~Gesmundo, C.~Ikenmeyer, and V.~Lysikov.
\newblock {Algebraic metacomplexity and representation theory}.
\newblock In {\em 40th Comp. Compl. Conf. (CCC 2025)}, volume 339 of {\em
  Leibniz Int. Proc. Inf. (LIPIcs)}, pages 26:1--26:35. Schloss Dagstuhl --
  Leibniz-Zentrum f{\"u}r Informatik, 2025.
\newblock \href {https://doi.org/10.4230/LIPIcs.CCC.2025.26}
  {\path{doi:10.4230/LIPIcs.CCC.2025.26}}.

\bibitem[Wey03]{Wey:VB}
J.~Weyman.
\newblock {\em {Cohomology of vector bundles and syzygies}}, volume 149 of {\em
  {Cambridge Tracts in Mathematics}}.
\newblock Cambridge University Press, Cambridge, 2003.

\end{thebibliography}
}

\end{document}